%% file: CombBndEF.tex
\documentclass[11pt,twoside,reqno]{amsart}
\usepackage{hyperref}
\usepackage[a4paper]{geometry}
\usepackage{graphicx}
\usepackage{color}
\usepackage{amssymb,amsfonts}
\usepackage{epstopdf}
\usepackage{enumerate}
\usepackage{times}
\usepackage{nicefrac}
\usepackage{bbm}
\usepackage{epsfig}

\usepackage[normalem]{ulem}

\usepackage{amsmath}
\usepackage{amsthm}
\newtheorem{thm}{Theorem}
\numberwithin{thm}{section}
\newtheorem{lem}[thm]{Lemma}
\newtheorem{prop}[thm]{Proposition}
\newtheorem{cor}[thm]{Corollary}
\newtheorem{conj}[thm]{Conjecture}
\theoremstyle{remark}
\newtheorem{rem}[thm]{Remark}

\newcommand{\conv}{\mathop{\mathrm{conv}}}
\newcommand{\vertexset}{\mathop{\mathrm{vert}}}
\newcommand{\nnegrank}{\mathop{\mathrm{rank}_+}} 
\newcommand{\rk}{\mathop{\mathrm{rank}}}
\newcommand{\rc}{\mathop{\mathrm{rc}}} 
\newcommand{\supp}{\mathop{\mathrm{supp}}} 
\newcommand{\suppmat}{\mathop{\mathrm{suppmat}}} 

\newcommand{\RR}{\mathbb{R}}

\newcommand{\lt}{\left}
\newcommand{\rt}{\right}

\renewcommand{\le}{\leqslant}
\renewcommand{\ge}{\geqslant}

\newcommand{\thegraph}{rectangle graph}
\newcommand{\theG}[1]{G(#1)}
\newcommand{\mnfx}{extension complexity}
\newcommand{\xc}[1]{\mathop{\mathrm{xc}}(#1)}
\newcommand{\nvtx}[1]{\lt\lvert#1\rt\rvert}

\DeclareMathOperator{\POp}{P}
\DeclareMathOperator{\matchOp}{match}
\newcommand{\PMatch}[1]{\POp_{\matchOp}({#1})}
\DeclareMathOperator{\stabOp}{stab}
\newcommand{\PStab}[1]{\POp_{\stabOp}({#1})}
\DeclareMathOperator{\edgeOp}{edge}
\newcommand{\PEdge}[1]{\POp_{\edgeOp}({#1})}
\newcommand{\ints}[1]{[{#1}]}

\newcommand{\R}{\mathbb{R}}
\newcommand{\facLat}[1]{\mathcal{L}({#1})}
\newcommand{\setDef}[2]{\{{#1}\,:\,{#2}\}}
\newcommand{\RNonNeg}{\R_+}
\newcommand{\unitVec}[1]{\mathbbm{e}_{#1}}

\DeclareMathOperator{\affOp}{aff}
\newcommand{\aff}[1]{\affOp({#1})}

\newcommand{\scalProd}[2]{\langle{#1},{#2}\rangle}

\begin{document}

\title{Combinatorial Bounds on Nonnegative Rank and Extended Formulations}%
\author[Fiorini]{Samuel Fiorini}%
\address{Samuel Fiorini: Department of Mathematics, Universit\'e libre de Bruxelles, Belgium}
\email{sfiorini@ulb.ac.be}%
\author[Kaibel]{Volker Kaibel}%
\address{Volker Kaibel, Kanstantsin Pashkovich, Dirk Oliver Theis: Faculty of Mathematics, Otto von Guericke Universit\"at Magdeburg, Germany}
\email{kaibel@ovgu.de,pashkovi@mail.math.uni-magdeburg.de,theis@ovgu.de}%
\author[Pashkovich]{Kanstantsin Pashkovich}%
\author[Theis]{Dirk Oliver Theis}%

\date{\today}%

%

\begin{abstract}
	An extended formulation of a polytope~$P$ is a system of linear inequalities and equations that describe some polyhedron which can be projected onto~$P$.  Extended formulations of small size (i.e., number of inequalities)  are of interest, as they allow to model corresponding optimization problems as linear programs of small sizes.  In this paper, we describe several aspects and new results on the main known approach to establish lower bounds on the sizes of extended formulations, which is to bound  from below the number of rectangles needed to cover the support of a slack matrix of the polytope. Our main goals are to shed some light on the question  how this combinatorial rectangle covering bound compares to other bounds known from the literature, and to obtain a better idea of the power as well as of the limitations of this bound.
	In particular, we provide geometric interpretations (and a slight sharpening) of Yannakakis'~\cite{Yannakakis91} result on the relation between  minimal sizes of  extended formulations and the nonnegative rank of slack matrices, and we describe the fooling set bound on the nonnegative rank (due to Dietzfelbinger et al.~\cite{DietzfelbingerHromkovicSchnitger96})  as the clique number of a certain graph. Among other results, we prove that both the cube as well as the Birkhoff polytope do not admit extended formulations with fewer inequalities than these polytopes have facets, and we show that every extended formulation of a $d$-dimensional neighborly polytope with $\Omega(d^2)$ vertices has size~$\Omega(d^2)$. 
\end{abstract}


\renewcommand{\thesubsubsection}{\thesubsection.(\alph{subsubsection})}      
\setcounter{tocdepth}{7}                                                     

\setcounter{page}{1}

\maketitle




\section{Introduction}

An \emph{extended formulation} of a polytope\footnote{For all basic concepts and results on polyhedra and polytopes we refer
  to~\cite{Ziegler}.}~$P\subseteq\R^d$ is a system of linear inequalities and linear equations defining a polyhedron~$Q\subseteq\R^e$ along with an affine map
$\pi:\R^e\rightarrow\R^d$ satisfying $\pi(Q)=P$. Note that by translating $P$ and/or $Q$, we may safely assume that $\pi$ is linear. The \emph{size} of the
extended formulation is the number of inequalities in the system (the number of equations is ignored, since one can easily get rid of them). As in this setting
linear optimization over~$P$ can be done by linear optimization over~$Q$, one is interested in finding small (and simple) extended formulations.
Fig.~\ref{fig:8gon} illustrates an extended formulation of size~$6$ in~$\R^3$ of a regular $8$-gon.
\begin{figure}[ht]
	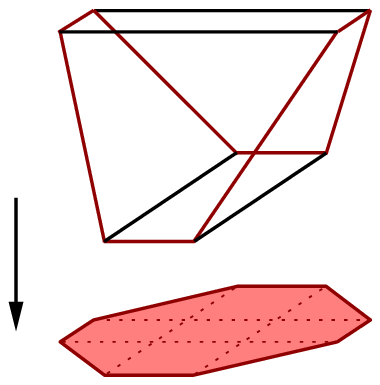
	\caption{A regular $8$-gon as a projection of a cube.\label{fig:8gon}}
\end{figure}
In fact, every regular $2^k$-gon has an extended formulation of size~$2k$~\cite{Ben-TalNemirovski01,KaibelPashkovich11}, which can, e.g., be exploited in order
to approximate second-order cones.

However, several polytopes associated with combinatorial optimization problems have surprisingly small extended formulations (for recent survey articles we
refer to~\cite{ConfortiCornuejolsZambelli10,VanderbeckWolsey2010,Kaibel11}). Among the nicest examples are extended formulations of size $O(n^3)$ for the
spanning tree polytopes of complete graphs with~$n$ nodes (due to Martin~\cite{KippMartin91}) and of size $O(n\log n)$ for the permutahedron associated with the
permutations of~$n$ elements (due to Goemans~\cite{Goemans09}).
It may not be very surprising that no polynomial size extended formulations of polytopes associated with NP-hard optimization problems like the traveling
salesman polytope are known.
However, the same holds, e.g., for matching polytopes to this day.  While it may well be that we simply still have missed the right techniques to construct
polynomial size extended formulations for the latter polytopes,
most people would probably agree that for the first ones we simply still miss the right techniques to prove lower bounds on the sizes of extended formulations
of concrete polytopes.
Note, however, that Rothvo\ss~\cite{Rothvoss11} recently established, by an elegant counting argument, the existence of 0/1-polytopes $P$ in
$\mathbb{R}^d$ such that the size of \emph{every} extended formulation of $P$ is exponential in $d$.

There is a beautiful approach due to Yannakakis~\cite{Yannakakis91} for deriving lower bounds on the sizes of extended formulations that share symmetries of the
polytope to be described. In fact, Yannakakis proved that neither the traveling salesman polytopes nor the matching polytopes associated with complete graphs
admit polynomial size extended formulations that are invariant under permuting the nodes of the graph. These techniques can, for instance, be extended to prove
that the same holds for the polytopes associated with cycles or matchings of logarithmic size in complete graphs, which do, however, have non-symmetric extended
formulations of polynomial size~\cite{KaibelPashkovichTheis10}, as well as to show that the permutahedra do not have symmetric extended formulations of
sub-quadratic size~\cite{Pashkovich09}.

Asking for the smallest size~$\xc{P}$ of an arbitrary extended formulation of a polytope~$P$, one finds that it suffices to consider extended formulations that
define \emph{bounded} polyhedra~$Q$, i.e., polytopes (see~\cite{FaenzaFioriniGrappeTiwary11}, \cite{Kaibel11}). Indeed, if~$Q$ is an unbounded
polyhedron that is mapped to the polytope~$P$ by the linear map~$\pi$, then the recession cone of~$Q$ is contained in the kernel of~$\pi$ (as~$\pi(Q)=P$ is
bounded). Thus, the pointed polyhedron~$Q'=Q\cap L$ with~$L$ being the orthogonal complement of the lineality space of~$Q$ satisfies $\pi(Q')=P$ as well. If $Q'$ is bounded, we are done. Otherwise, choosing some vector~$a$ that satisfies $\scalProd{a}{x}>0$ for all non-zero elements of the recession cone of~$Q'$ and some $\beta\in\R$ such that
$\beta>\scalProd{a}{v}$ holds for all vertices~$v$ of~$Q'$, we have that $Q'' := \setDef{x\in Q'}{\scalProd{a}{x}=\beta}$ is a polytope with
$\pi(Q'')=P$.
  
Therefore, defining an \emph{extension} of a polytope~$P$ to be a \emph{polytope}~$Q$ along with an affine projection that maps~$Q$ to~$P$, the extension complexity~$\xc{P}$ of~$P$ equals the minimal number of facets of all extensions of~$P$ (where we can even restrict our attention to full dimensional extensions). Note that since a polytope is the set of all convex combinations of its vertices, it has a simplex with the same number of vertices as an extension, showing that~$\xc{P}$ is at most the minimum of the numbers of vertices and facets of~$P$. 

In the same paper~\cite{Yannakakis91} where he established the above mentioned lower bounds on \emph{symmetric} extended formulations for matching and traveling
salesman polytopes, Yannakakis essentially also showed that the extension complexity of a polytope~$P$ equals the \emph{nonnegative rank} of a \emph{slack
  matrix} of~$P$, where the latter is a nonnegative matrix whose rows and columns are indexed by the facets and vertices of~$P$, respectively, storing in each
row the slacks of the vertices in an inequality defining the respective facet, and the nonnegative rank of a nonnegative matrix~$M$ is the smallest number~$r$
such that~$M$ can be written as a product of two nonnegative matrices with~$r$ columns and $r$~rows, respectively.
(See \cite{LauMarkham78}, or \cite{BeasleyLaffey09} and the references therein.)
Such a nonnegative factorization readily induces a covering of the set~$\supp(M)$ of non-zero positions of~$M$ by~$r$ \emph{rectangles}, i.e., sets formed as
the cartesian product of subsets of the row and of the column indices. Thus, the smallest number of rectangles that cover~$\supp(M)$ (the \emph{rectangle
  covering number} of~$M$) is a lower bound on the nonnegative rank of~$M$, and hence, on the extension complexity of~$P$ if~$M$ is the slack matrix of~$P$.
Rectangle covering is the same concept as covering the edges of a bipartite graph by bicliques (e.g.,~\cite{DohertyLundgrenSiewert99}), and it also coincides
with factorizing a Boolean matrix (e.g., \cite{CaenGregoryPullman81,GregoryPullman83}).

In fact, almost all techniques to find lower bounds on the extension complexity of polytopes are combinatorial in the sense that they yield lower bounds on the
rectangle covering numbers of slack matrices rather than exploiting the true numbers in the matrices. The main goal of this paper is to present current
knowledge on the concept of rectangle covering numbers in particular with respect to its relation to the extension complexity of polytopes and to add several
results both on lower bounds on rectangle covering number (providing lower bounds on the extension complexity of certain polytopes) as well as on upper bounds
(revealing limitations of this approach to find lower bounds on the extension complexity of polytopes). Though the paper is written from the geometric point of
view, we hope that it may also be useful for people working on questions concerning the nonnegative rank from a more algebraic point of view.

The paper starts by describing (in Section~\ref{sec:projAndDerive}) the relation between rectangle coverings and extensions both in the pure combinatorial
setting of embeddings of face lattices as well as via the characterization of the extension complexity as the nonnegative rank of slack matrices.  Here, we also
give geometric explanations of Yannakakis' algebraic result relating extensions and nonnegative factorizations. We then give some results on upper bounds on
rectangle covering numbers. For instance, we show that the rectangle covering number of slack matrices of polytopes with~$n$ vertices and at most~$k$ vertices
in every facet is bounded by $O(k^2\log n)$ (Proposition~\ref{prop:UB_few_vtcs_per_facet}). In Section~\ref{sec:CC} we briefly describe the connection between
the rectangle covering number of a nonnegative matrix and deterministic as well as nondeterministic communication complexity. While the relation to
deterministic communication complexity theory had already been used by Yannanakis~\cite{Yannakakis91}, recent results of Huang and
Sudakov~\cite{HuangSudakov2010} on nondeterministic communication complexity have interesting implications for questions concerning extension complexity as
well. In Section~\ref{sec:lb} we then present several techniques to derive lower bounds on rectangle covering numbers, as well as some new results obtained by
using these techniques. We start by proving that the logarithm of the number of faces of a polytope~$P$ is not only a lower bound on its extension complexity
(as already observed by Goemans~\cite{Goemans09}) but also on the rectangle number of a slack matrix of~$P$. Observing that the rectangle covering number of a
nonnegative matrix is the chromatic number of the \emph{rectangle graph} of the matrix, we interpret the \emph{fooling set bound} (see Dietzfelbinger et
al.~\cite{DietzfelbingerHromkovicSchnitger96}) as the clique number of that graph. Exploiting that bound, we show, for instance, that for both the cube as well
as for the Birkhoff polytope the extension complexity equals the number of facets (Propositions~\ref{prop:lb:clique:cube} and~\ref{prop:lb:birhoff}). However,
we also show that for no $d$-dimensional polytope the clique bound can yield a better lower bound on the extension complexity than $(d+1)^2$. We also treat (in
Section~\ref{ssec:lb:alpha-goof}) the relation between the independence ratio of the rectangle graph and the concept of \emph{generalized fooling sets} (due to
Dietzfelbinger et al.~\cite{DietzfelbingerHromkovicSchnitger96}), and we prove that for $d$-dimensional neighborly polytopes with $\Omega(d^2)$ vertices the
rectangle covering number (and thus the extension complexity) is bounded from below by $\Omega(d^2)$.

We close this introduction with a few remarks on notation. We only distinguish between row and column vectors in the context of matrix multiplications, where, as usual, vectors are meant to be understood as column vectors.  For a matrix~$M$, we denote by $M_{i,\star}$ and $M_{\star,j}$ the vectors in the $i$-th row and in the $j$-th column of~$M$, respectively. All logarithms are meant to refer to base two.

\section{Projections and Derived Structures}\label{sec:projAndDerive}


\subsection{Lattice Embeddings}

%


Let us denote by~$\facLat{P}$ the \emph{face lattice} of a polytope~$P$, i.e., the set of faces of~$P$ (including the \emph{non-proper} faces $\varnothing$ and~$P$ itself) partially ordered by inclusion.  
The following proposition describes the relation between the face lattices of two polytopes one of which is a projection of the other. The statement certainly is not new, but as we are not aware of any explicit reference for it, we include a brief proof. An \emph{embedding} of a partially ordered set $(S,\leqslant)$ into a partially ordered set $(T,\sqsubseteq)$ is a map $f$ such that $u \leqslant v$ if and only if $f(u) \sqsubseteq f(v)$. Notice that every embedding is injective.

\begin{prop}\label{prop:embeddingFacLat}
 If~$Q\subseteq\R^e$ and~$P\subseteq\R^d$ are two polytopes, and~$\pi:\R^e\rightarrow\R^d$ is an affine map with $\pi(Q)=P$, then the map that assigns $h(F) := Q \cap\pi^{-1}(F)$ to each face~$F$ of~$P$ defines an embedding of~$\facLat{P}$ into~$\facLat{Q}$.
\end{prop}

\begin{proof}
 For any face~$F$ of~$P$ defined by an inequality $\scalProd{a}{x}\leqslant\beta$, we have $\scalProd{a}{\pi(y)} \leqslant \beta$ for all $y\in Q$ with equality if and only if $\pi(y) \in F$ holds. Thus, $h(F)$ is indeed a face of~$Q$, defined by the linear inequality $\scalProd{a}{\pi(y)} \leqslant \beta$. Moreover, $h$ is clearly an embedding. 
\end{proof}

\begin{rem}
 From the definition of $h$, we see that $h(F\cap G) = h(F)\cap h(G)$ holds for all faces $F,G$. 
We call a lattice embedding $h$ with this property \emph{meet-faithful}. Clearly, not all lattice embeddings are meet-faithful.
\end{rem}

Figure~\ref{fig:embedding} illustrates the embedding from Proposition~\ref{prop:embeddingFacLat}. In the figure, $P$ is a $4$-dimensional cross-polytope, $Q$ is a $7$-dimensional simplex and $\pi$ is any affine projection mapping the $8$ vertices of $Q$ to the $8$ vertices of $P$. Denoting by $\unitVec{i}$ the $i$th unit vector, we have $Q = \conv \{0,\unitVec{1},\ldots,\unitVec{7}\} \subseteq \mathbb{R}^7$, $P = \conv\{\unitVec{1},-\unitVec{1},\dots,\unitVec{4},-\unitVec{4}\} \subseteq \mathbb{R}^4$ and for instance $\pi(0) = \unitVec{1}$, $\pi(\unitVec{1}) = -\unitVec{1}$, \ldots, $\pi(\unitVec{7}) = -\unitVec{4}$. As the figure suggests correctly, constructing a small extended formulation for a polytope~$P$ means to hide the facets of~$P$ in the fat middle part of the face lattice of an extension with few facets.

\begin{figure}[ht]
 \centerline{\includegraphics[height=10cm]{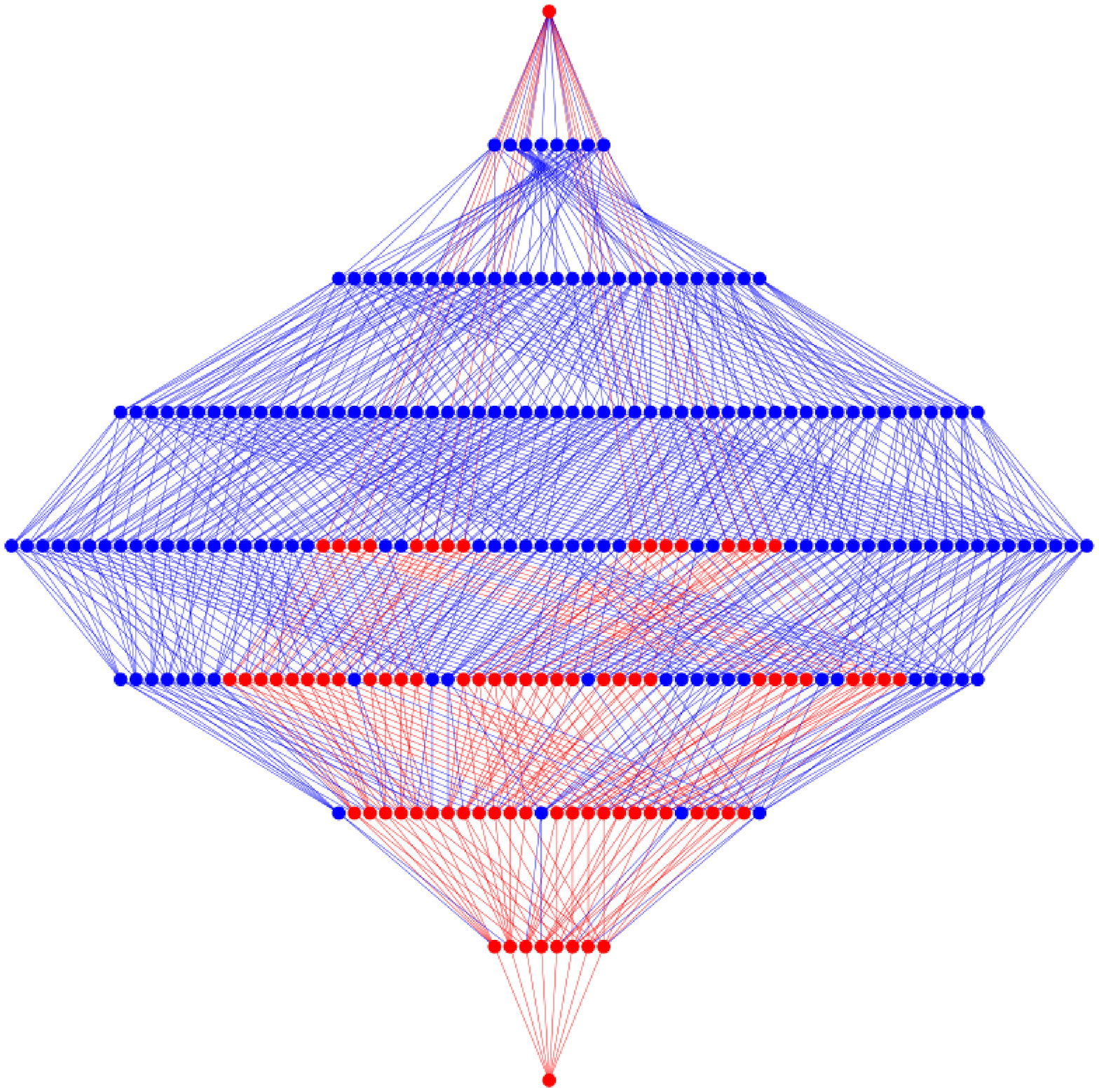}}
 \caption{Embedding of the face lattice of the $4$-dimensional cross-polytope into the face lattice of the $7$-dimensional simplex.\label{fig:embedding}}
\end{figure}

\subsection{Slack Representations}

Let~$P$ be a polytope in $\R^d$. Let $Ax \leqslant b$ be a system of linear inequalities such that
\begin{equation*}
P = \setDef{x \in \R^d}{Ax \leqslant b}\,,
\end{equation*}
and let $m$ denote the number of inequalities involved in  $Ax \leqslant b$. The \emph{slack vector} of a point $x \in \R^d$ is the vector $b - Ax \in \R^m$, and  the \emph{slack map} of~$P$ w.r.t.\ $Ax \leqslant b$ is  the affine map $\sigma: \R^d \rightarrow \R^m$ that maps each vector to its slack vector. We call the polytope $\tilde{P} := \sigma(P)$, which is the image of $P$ under the slack map $\sigma$, the \emph{slack representation} of $P$ w.r.t.\ $Ax \leqslant b$. Clearly the affine hull $\aff{\tilde{P}}$ of $\tilde{P}$ is the image of the affine hull $\aff{P}$ of $P$ under $\sigma$. We call $\aff{\tilde{P}}$ the \emph{slack space} of $P$ w.r.t.\ $Ax \leqslant b$. Note that the polytope $\tilde{P}$ is affinely isomorphic to the polytope $P$ and $\tilde{P} = \aff{\tilde{P}} \cap \RNonNeg^m$. 

 A set of $r$ nonnegative vectors $T := \{t_1,\ldots,t_r\} \subseteq \RNonNeg^m$ is called a \emph{slack generating set} of~$P$ w.r.t.\ $Ax \leqslant b$ of size $r$ if  every point in $\tilde{P}$ can be expressed as a nonnegative combination of the vectors in~$T$. Then the following system  
\begin{equation*}
	\sum_{k = 1}^r \lambda_k t_k\in\aff{\tilde{P}} \quad\text{and}\quad  \lambda_k\geqslant 0 \text{ for all }k \in [r]
\end{equation*}
provides an extended formulation of~$\tilde{P}$ of size~$r$ via the projection map $\lambda\mapsto\sum_{k=1}^r \lambda_k t_k$. Since $P$ and $\tilde{P}$ are affinely isomorphic, this also yields an extension of~$P$ of size~$r$, which is called a \emph{slack extension} of $P$ w.r.t.\ $Ax \leqslant b$.


\begin{lem}
\label{lem:dim(P)_at_least_1}
If the inequality $\scalProd{c}{x} \leqslant\delta$ is valid for a polytope 
$P := \{x \in \R^d \mid Ax \leqslant b\}$ with $\dim(P) \geqslant 1$, then it is a nonnegative combination of the inequalities of the system $Ax \leqslant b$ defining $P$.
\end{lem}
\begin{proof}
By Farkas's lemma, because $P \neq \varnothing$, there exists $\delta' \leqslant \delta$ such that $\scalProd{c}{x} \leqslant \delta'$ is a nonnegative combination of the inequalities of $Ax \leqslant b$. Now let $i$ be an index such that $\min \{\scalProd{A_{i,\star}}{x} \mid x \in P\} < \max \{\scalProd{A_{i,\star}}{x} \mid x \in P\} = b_i$. Such an index exists because $\dim(P) \geqslant 1$. Letting $b'_i := \min \{\scalProd{A_{i,\star}}{x} \mid x \in P\}$, we see that the inequality $-\scalProd{A_{i,\star}}{x} \leqslant -b'_i$ is valid for $P$. By Farkas's lemma, because of the minimality of $b'_i$, this last inequality is a nonnegative combination of the inequalities of $Ax \leqslant b$. By adding the inequality $\scalProd{A_{i,\star}}{x} \leqslant b_i$ and then scaling by $(b_i-b'_i)^{-1}$, we infer that $0 \leqslant 1$ is a nonnegative combination of the inequalities of $Ax \leqslant b$. After scaling by $\delta-\delta'$ and adding the inequality $\scalProd{c}{x} \leqslant \delta'$, we conclude that $\scalProd{c}{x} \leqslant \delta$ is a nonnegative combination of the inequalities of $Ax \leqslant b$.
\end{proof}

\begin{lem}
\label{lem:slack_gen}
Let $P := \{x \in \R^d \mid Ax \leqslant b\}$ be a polytope with $\dim(P) \geqslant 1$. If the polytope $P$ has $f$ facets, then $P$ has a slack generating set of size $f$ w.r.t.\ $Ax \leqslant b$.
\end{lem}
\begin{proof}
Let  $m$ and $r$ denote the total number and the number of nonredundant inequalities of $Ax \leqslant b$ respectively. W.l.o.g., assume that these $r$ nonredundant inequalities come first in the system $Ax \leqslant b$. By Lemma \ref{lem:dim(P)_at_least_1}, each of the $m-r$ last inequalities of $Ax \leqslant b$ is a nonnegative combination of the $r$ first inequalities of $Ax \leqslant b$. This implies that, for each $i \in [m] \setminus [r]$, there are nonnegative coefficients $t_{k,i} \in \RNonNeg$ such that
\begin{equation*}
A_{i,\star} = \sum_{k = 1}^r A_{k,\star} t_{k,i} \quad 
\text{and} \quad 
b_i = \sum_{k=1}^r b_k t_{k,i}\,.
\end{equation*}
By letting $t_{k,i} := 1$ if $i = k$ and $t_{k,i} := 0$ for $i \in [r]$, the above equations hold for all $i \in [m]$. This defines a set $T := \{t_1, \ldots, t_r\}$ of nonnegative vectors $t_k \in \RNonNeg^m$, $k \in [r]$.

To show that $T$ is a slack generating set, consider for any point $x \in P$ the nonnegative combination%
\begin{equation}
\label{eq:nneg_comb_slack}
\sum_{k=1}^r (b_k - \scalProd{A_{k,\star}}{x}) t_k\,.
\end{equation}
For $i \in [m]$, the $i$-th coordinate of this vector equals
\begin{equation*}
\sum_{k=1}^r (b_k - \scalProd{A_{k,\star}}{x}) t_{k,i}
= \sum_{k=1}^r b_k t_{k,i} - \scalProd{\sum_{k=1}^r A_{k,\star} t_{k,i}}{x}
= b_i - \scalProd{A_{i,\star}}{x}, 
\end{equation*}
thus \eqref{eq:nneg_comb_slack} equals the slack vector of $x$. 

Moreover, in \eqref{eq:nneg_comb_slack} the coefficients corresponding to inequalities that hold with equality for all points of $P$ can be chosen to be  zeroes. Thus the corresponding vectors can be removed from $T$, resulting in a slack generating set of size $f$.
\end{proof}

\begin{thm}
	\label{thm:extCmplSlackExt}
	The \mnfx{} of a polytope~$P$ with $\dim(P) \geqslant 1$ is equal to the minimum size of a slack generating set of~$P$ w.r.t.\ any given system $Ax\leqslant b$ defining $P$.
\end{thm}
\begin{proof}
 For an extension~$Q \subseteq \R^e$ of~$P$ of size~$r$  we have to show  that there is a slack generating set of the polytope $P$ of size $r$. Since the polytopes $P$ and $\tilde{P}$ are affinely isomorphic, $Q$ is an extension of the polytope~$\tilde{P}$, i.e there exists an affine map $\pi : \R^e \to \R^m$ such that $\pi(Q) = \tilde{P}$. Let $c_i \in \R^e$, $i \in [m]$ and $g_i \in \R$ be such that $\pi(y)_i = g_i - \scalProd{c_i}{y}$. Note that the inequality $\scalProd{c_i}{y} \leqslant g_i$ is valid for $Q$, because  $\pi(Q) \subseteq \tilde{P} \subseteq \RNonNeg^m$. 

Now, consider a system $Cy \leqslant g$ of $n$ linear inequalities describing $Q$ such that the first $m$ inequalities of this system are $\scalProd{c_1}{y} \leqslant g_1$, \ldots, $\scalProd{c_m}{y} \leqslant g_m$, and a corresponding slack map $\tau : \R^e \to \R^n$. Again, $\tilde{Q} := \tau(Q)$ is affinely isomorphic to $Q$. Therefore, $\tilde{Q}$ is also an extension of $\tilde{P}$. Indeed, the projection $\rho : \R^n \to \R^m$ to the first $m$ coordinates maps $\tilde{Q}$ onto $\tilde{P}$, because $\pi = \rho \circ \tau$ (see Figure \ref{fig:diagram}). 

\begin{figure}[ht]
\begin{equation*}
\begin{array}{ccc}
Q &\stackrel{\tau}{\longrightarrow} &\tilde{Q}\\[1ex]
  &\searrow^{\!\!\pi} & \downarrow^{\rho}\\[1ex]
P &\stackrel{\sigma}\longrightarrow & \tilde{P}
\end{array}
\end{equation*}
\caption{The four polytopes $P$, $\tilde{P}$, $Q$ and $\tilde{Q}$ and the maps between them.\label{fig:diagram}}
\end{figure}

By Lemma \ref{lem:slack_gen} there is a slack generating set $\{t_1,\ldots,t_r\} \subseteq \RNonNeg^n$ for $Q$ of size $r$. We claim that stripping the last $n - m$ coordinates of these vectors result in a slack generating set $\{\rho(t_1),\ldots,\rho(t_r)\}$ of the polytope $P$.

Indeed, consider a point $x \in P$ and its slack vector $\sigma(x) \in \tilde{P}$. Because $\tilde{P} \subseteq \pi(Q)$ there is a point $y \in Q$ such that $\pi(y) = \sigma(x)$. Because $\{t_1,\ldots,t_r\}$ is a slack generating set of $Q$, there exist nonnegative coefficients $\lambda_1, \ldots, \lambda_r \in \RNonNeg$ such that $\tau(y) = \sum_{k=1}^r \lambda_k t_k$. From linearity of $\rho$ and $\pi = \rho \circ \tau$, we have
$$
\sigma(x) = \pi(y) = \rho(\tau(y)) = \rho\left(\sum_{k=1}^r \lambda_k t_k\right) = \sum_{k=1}^r \lambda_k \rho(t_k)\,.
$$
 and thus $P$ has a slack generating set $\{\rho(t_1),\ldots,\rho(t_r)\}$ of size $r$.
\end{proof}

Note that in Theorem~\ref{thm:extCmplSlackExt} one may take the minimum over the  slack generating sets w.r.t.\ \emph{any} fixed system of inequalities describing~$P$. In particular, all these minima coincide.  


\subsection{Non-Negative Factorizations}

Again, let $P=\setDef{x\in\R^d}{Ax\leqslant b}$ be a polytope, where $A\in\R^{m\times d}$, $b\in\R^m$. Let $V = \{v_1,\ldots,v_n\}\subseteq \R^d$ denote any finite set such that
\begin{equation*}
	P = \conv(V)\,.
\end{equation*}
The \emph{slack matrix} of~$P$ w.r.t.~$Ax\leqslant b$ and $V$ is the matrix $S=(S_{i,j})\in\RNonNeg^{m\times n}$ with 
\begin{equation*}
	S_{i,j}=b_i-\scalProd{A_{i,\star}}{v_j}\quad\text{for all }i\in\ints{m},\ j\in \ints{n}\,.
\end{equation*}
Note that the slack representation~$\tilde{P}\subseteq\R^m$ of~$P$ (w.r.t. $Ax\leqslant b$) is the convex hull of the columns of~$S$.

If the columns of a nonnegative matrix $T \in \RNonNeg^{m\times r}$ form a slack generating set of~$P$, then there is a nonnegative matrix $U\in\RNonNeg^{r\times n}$ with $S=TU$. Conversely, for every  factorization $S=TU$ of the slack matrix into nonnegative matrices $T\in\RNonNeg^{m\times r}$ and $U\in\RNonNeg^{r\times n}$, the columns of~$T$ form a slack generating set of~$P$. 

Therefore, due to Theorem~\ref{thm:extCmplSlackExt}, constructing an extended formulation of size~$r$ for~$P$ amounts to finding a \emph{rank-$r$} nonnegative factorization of the slack matrix~$S$, that is a factorization $S=TU$ into nonnegative matrices~$T$ with~$r$ columns and~$U$ with~$r$ rows. In particular, the following result follows, which is essentially due to Yannakakis~\cite{Yannakakis91} (he proved that \mnfx{} and nonnegative rank are within a factor of two of each other, when the size of an extension is defined as the sum of the number of variables and number of constraints defining the extension).

\begin{thm}[see Yannakakis~\cite{Yannakakis91}]
	\label{thm:extComplNonnegRk}
The \mnfx{} of a  polytope~$P$ (which is neither empty nor a single point) is equal to the nonnegative rank of any of its slack matrices. In particular, all the slack matrices of~$P$ have the same nonnegative rank.
\end{thm}

\begin{rem}
  It is obvious from the definition that the extension complexity is monotone on extensions: if polytope $Q$ is an extension of polytope $P$, then $\xc{Q} \geqslant \xc{P}$. From Theorem \ref{thm:extComplNonnegRk}, we see immediately that the extension complexity is also monotone on faces: if polytope $P$ has $F$ as a face, then $\xc{P} \geqslant \xc{F}$ because we can obtain a slack matrix of $F$ from a slack matrix of $P$ by deleting all columns that correspond to points of $P$ which are not in $F$.
\end{rem}

A nonnegative matrix~$T$ is the first factor in a nonnegative factorization of a slack matrix of~$P$ if and only if the columns of~$T$ form a slack generating set of $P$. In order to characterize the second factors of such nonnegative factorizations, let us consider an extension~$Q$ of~$P$ with projection~$\pi$ and some set~$V = \{v_1,\ldots,v_n\}$ with $P=\conv(V)$. A \emph{section} is a map $s : V\rightarrow Q$ such that $\pi(s(x))=x$ holds for all $x\in V$. Clearly, every extension posseses a section map. 

Recall that $(\R^n)^*$ denotes the dual vector space of $\R^n$, which we here regard as the set of all row vectors of size $n$. For an inequality $\scalProd{c}{y} \leqslant g$ that is valid for~$Q$, we construct a nonnegative row vector in $(\R^n)^*_+$ (the nonnegative orthant in $(\R^n)^*$), whose $j$-th coordinate equals $g-\scalProd{c}{s(v_j)}$ and call it the \emph{section slack covector} associated with $\scalProd{c}{y} \leqslant g$ w.r.t.\ set $V$. A set $U$ of nonnegative row vectors in $(\R^n)^*_+$ is a \emph{complete set of section slack covectors} if there is some extended formulation for~$P$ along with some section such that~$U$ is precisely the set of section slack covectors associated with the inequalities in the extended formulation. A nonnegative matrix~$U$ is the second factor in a nonnegative factorization of a slack matrix of~$P$ if and only if the rows of~$U$ form a complete set of section slack covectors (both w.r.t. to the same set~$V$ with $P=\conv(V)$). From this, one in particular derives (again using Farkas's Lemma) the following characterization, where the \emph{slack covector} associated with some valid inequality $\scalProd{a}{x} \leqslant b$ for~$P=\conv(V)$ is the nonnegative row vector in $(\R^n)^*_+$ whose $j$-th coordinate is $b-\scalProd{a}{v_j}$.

\begin{prop}
	A set $U\subseteq(\R^n)^*_+$ is a complete set of section slack covectors for $P=\conv(V)$ (with $1\le|V|<\infty$) if and only if every slack covector associated with a valid inequality for~$P$ can be expressed as a nonnegative combination of the elements of $U$.
\end{prop}

\subsection{Rectangle Coverings}

According to Theorem~\ref{thm:extComplNonnegRk}, finding lower bounds on the \mnfx{} of a polytope amounts to finding lower bounds on the nonnegative rank of its slack matrices. Not surprisingly, determining the nonnegative rank  of a matrix is a hard problem from the algorithmic point of view. Indeed, it is NP-hard to decide whether the nonnegative rank of a matrix equals its usual rank~\cite{Vavasis09} (where, of course, the first is never  smaller than the second). 

One way to bound~ the nonnegative rank $\nnegrank(M)$ from below is to observe that a nonnegative factorization $M=TU$, where~$T$ has~$r$ columns and~ $U$ has~$r$ rows, yields a representation 
\begin{equation*}
	M=\sum_{k=1}^r T_{\star,k}U_{k,\star}
\end{equation*}
of~$M$ as a sum of~$r$ nonnegative rank-$1$ matrices $T_{\star,k}U_{k,\star}$. Denoting by $\supp(\cdot)$ the \emph{support} of a vector or of a matrix (i.e., the subset of indices where the argument has a nonzero entry), we obviously have 

\begin{equation*}
	\supp(M)=\bigcup_{k=1}^r\supp(T_{\star,k}U_{k,\star})
\end{equation*}
since both~$T$ and~$U$ are nonnegative.

A \emph{rectangle} is a set of the form $I\times J$, where~$I$ and~$J$ are subsets of the row respectively column indices of~$M$. A \emph{rectangle covering} of~$M$ is a set of rectangles whose union equals~\emph{$\supp(M)$}. It is important to notice that all the rectangles in any rectangle covering of~$M$ are contained in the support of $M$. The \emph{rectangle covering number} of~$M$ is the smallest cardinality~$\rc(M)$ of any rectangle covering of~$M$. Clearly, we have 

\begin{equation}\label{eq:rcnnegrank}
	\rc(M) \leqslant\nnegrank(M)
\end{equation}
for all nonnegative matrices~$M$ since
\begin{equation*}
	\supp(T_{\star,k}U_{k,\star})=\supp(T_{\star,k})\times \supp(U_{k,\star})
\end{equation*}
holds for each $k\in\ints{r}$ . In particular, when~$M = S$ is some slack matrix of the polytope~$P$ (neither empty nor a single point), then by Theorem~\ref{thm:extComplNonnegRk} we have
\begin{equation*}
	\rc(S)\leqslant\nnegrank(S)=\xc{P}\,.
\end{equation*}

Like~$\nnegrank(S)=\xc{P}$, the rectangle covering number~$\rc(S)$ is actually independent of the actual choice of the slack matrix~$S$ of~$P$, see Lemma~\ref{lem:rc_independent_of_S} below. 

Let us call the \emph{support matrix} $\suppmat(S)$ of~$S$ the 0/1-matrix arising from~$S$ by replacing all nonzero-entries by ones. Clearly, we have $\rc(S)=\rc(\suppmat(S))$. The rectangles that can be part of a rectangle covering of~$S$ are called \emph{$1$-rectangles} because any such rectangle induces a submatrix of $\suppmat(S)$ that contains only one-entries.

Furthermore, any 0/1-matrix whose rows are indexed by some set $F_1$, \ldots, $F_m$ of faces of~$P$ including all facets and whose columns are indexed by some set $G_1$, \ldots, $G_n$ of nonempty faces of~$P$ including all vertices such that there is a one-entry at position $(F_i,G_j)$ if and only if face~$F_i$ does \emph{not} contain face~$G_j$ is called a \emph{non-incidence matrix} for~$P$. Associating with every inequality in $Ax\leqslant b$ the face of~$P$ it defines and with every point in~$V$ the smallest face of~$P$ it is contained in, one finds that the set of support matrices of slack matrices of~$P$ equals the set of non-incidence matrices of~$P$ (up to  adding/removing repeated rows or columns).

Clearly, adding a row or a column to a nonnegative matrix does neither decrease the rectangle covering number nor the nonnegative rank. The following result on the rectangle covering number is easy to see.

\begin{lem}\label{lem:addrowcol}
	The rectangle covering number remains unchanged if one adds a row or a column whose support is the union of the supports of some existing rows or columns, respectively.
\end{lem}

Now, we can prove that all non-incidence matrices, and thus all slack matrices, of a polytope (of dimension at least $1$) have the same rectangle covering number.

\begin{lem}
\label{lem:rc_independent_of_S}
All non-incidence matrices of a polytope (which is neither empty nor a single point) have the same rectangle covering number.
\end{lem}

\begin{proof}
Letting~$P$ denote the polytope in the statement of the lemma, consider a non-incidence matrix $M$ for~$P$. Every proper face~$F$ of~$P$ contains exactly those faces of~$P$ that are contained in all facets of~$P$ that contain~$F$. For the non-incidence matrix~$M$, this implies that the support of the row corresponding to~$F$ is exactly the union of the supports of the rows corresponding to the facets containing~$F$.

Similarly, every nonempty face~$G$ of~$P$ is contained in exactly those faces of~$P$ in which all vertices of~$G$ are contained. This implies that the support of the column corresponding to~$G$ is exactly the union of the supports of the rows corresponding to the vertices contained in~$F$.

The result then follows from Lemma~\ref{lem:addrowcol} because the rectangle covering number of~$M$ equals the rectangle covering of the submatrix of~$M$ corresponding to the facets and vertices of~$P$, and thus does not depend on~$M$. (Notice that there might be a row in~$M$ for~$F = P$, but this row is identically zero.)
\end{proof}

Denoting by~$\rc(P)$ the rectangle covering number of any non-incidence matrix for~$P$, we have
\begin{equation}\label{eq:rcxc}
	\rc(P)\leqslant\xc{P}\,.
\end{equation}
When studying the rectangle covering number, we may freely choose the non-incidence matrix we consider. The most natural choice is the \emph{facet vs.\ vertex}
non-incidence matrix, since it appears as a submatrix of every non-incidence matrix. (The rows of this 0/1-matrix are indexed by the facets, and the columns are
indexed by the vertices. The entry corresponding to a facet-vertex pair is $1$ if and only if the facet does not contain the vertex.) Nevertheless, we will
sometimes consider non-incidence matrices with more rows.

\subsection{Rectangle Coverings, Boolean Factorizations and Lattice Embeddings}

As mentioned before, a rank-$r$ nonnegative factorization of a nonnegative matrix $M$ can be regarded as a decomposition of $M$ into a sum of $r$
nonnegative rank-$1$ matrices.  Similarly, we can regard a rectangle covering of a Boolean (or 0/1-) matrix $M$ with $r$ rectangles as a \emph{rank-$r$ Boolean
  factorization}, that is a factorization $M = TU$ expressing $M$ as the Boolean product of two Boolean matrices $T$ with $r$ columns and $U$ with $r$ rows.
Furthermore, a rank-$r$ Boolean factorization of a Boolean matrix $M$ is equivalent to an embedding of the relation defined by $M$ into the Boolean lattice
$2^{[r]}$ (that is, the set of all subsets of $[r]$, partially ordered by inclusion), in the sense of the following lemma.

\begin{lem}\label{lem:Boolean_factorization_iff_embedding}
  Let $M$ be a $m \times n$ Boolean matrix. Then $M$ admits a rank-$r$ Boolean factorization $M = TU$ if and only if there are functions $f\colon [m] \to 2^{[r]}$
  and $g\colon [n] \to 2^{[r]}$ such that $M_{ij} = 0$ if and only if $g(j) \subseteq f(i)$ for all $(i,j) \in [m] \times [n]$.
\end{lem}
\begin{proof}
  This follows immediately by interpreting the $i$th row of the left factor $T$ as the incidence vector of the \emph{complement} of the set $f(i)$ and the $j$th
  column of the right factor $U$ as the incidence vector of the set $g(j)$.
\end{proof}

From Lemma~\ref{lem:Boolean_factorization_iff_embedding}, we can easily conclude a first lattice-combinatorial characterization of the rectangle covering number
$\rc(P)$ of a polytope $P$ by taking $M$ to be a non-incidence matrix of $P$.

\begin{thm}\label{thm:lattice-embedding-Boolean}
  Let $P$ be a polytope with $\dim(P) \geqslant 1$. Then $\rc(P)$ is the smallest $r \geqslant 1$ such that $\mathcal{L}(P)$ embeds into $2^{[r]}$.
\end{thm}

For a given poset, the smallest number~$r$ such that~$P$ embeds into the Boolean poset~$2^{[r]}$ is known as the 2-dimension of the poset (see e.g.,
\cite{HabibNourineRaynaudThierry04} and the references therein).  We note that, independently, Gouveia et al.~\cite{GouveiaParriloThomas12} have found a similar
connection.

\begin{proof}[Proof of the Theorem~\ref{thm:lattice-embedding-Boolean}]
  Let $F_1$, \ldots, $F_m$ denote the facets of $P$, let $v_1$, \ldots, $v_n$ denote the vertices of $P$ and let $M$ denote the facet vs.\ vertex non-incidence
  matrix for $P$.

  Suppose first that $\rc(P) \leqslant r$, that is, $M$ has a rank-$r$ Boolean factorization. From maps $f$ and $g$ as in
  Lemma~\ref{lem:Boolean_factorization_iff_embedding}, we define a map $h$ from $\mathcal{L}(P)$ to $2^{[r]}$ by letting $h(F) := \bigcap_{i : F \subseteq F_i}
  f(i)$. It is clear that $F \subseteq G$ implies $h(F) \subseteq h(G)$. Now assume $h(F) \subseteq h(G)$. Pick a vertex $v_j$ of~$F$. We have $M_{ij} = 0$ and
  thus $g(j) \subseteq f(i)$ for all facets $F_i$ containing $F$. Hence, $g(j) \subseteq h(F)$. Because $h(F) \subseteq h(G)$, we have $g(j) \subseteq f(i)$ and
  thus $M_{ij} = 0$ for all facets $F_i$ containing~ $G$. Since $v_j$ is arbitrary, this means that every facet containing $G$ contains all vertices of $F$,
  which implies $F \subseteq G$. Therefore, $h$ is an embedding.

  Next, if $h$ is an embedding of $\mathcal{L}(P)$ into $2^{[r]}$, we can simply define $f(i) := h(F_i)$ and $g(j) := h(\{v_j\})$. These maps satisfy the
  conditions of Lemma~\ref{lem:Boolean_factorization_iff_embedding}, hence $M$ has a rank-$r$ Boolean factorization and $\rc(P) \leqslant r$. The theorem
  follows.
\end{proof}

We immediately obtain a second lattice-combinatorial characterization of the rectangle covering number of a polytope.

\begin{cor}\label{cor:lattice-embedding}
  The rectangle covering number~$\rc(P)$ of a polytope $P$ with $\dim(P) \geqslant 1$ is equal to the minimum number of facets of a polytope~$Q$ into whose face
  lattice~$\mathcal{L}(Q)$ the face lattice $\mathcal{L}(P)$ of~$P$ can be embedded.
\end{cor}

Corollary~\ref{cor:lattice-embedding} follows from Theorem~\ref{thm:lattice-embedding-Boolean} and Lemma~\ref{lem:trivial_embedding} below.

\begin{lem}\label{lem:trivial_embedding}
  If $Q$ is a polytope with $r \geqslant 2$ facets then there is an embedding of $\mathcal{L}(Q)$ into $2^{[r]}$.
\end{lem}
\begin{proof}
  We denote by $G_1$, \ldots, $G_r$ the facets of $Q$, and define a map $h$ from $\mathcal{L}(Q)$ to $2^{[r]}$ by letting $h(G) := \{k \in [r] \mid G
  \not\subseteq G_k\}$ for all faces $G$ of $Q$. It is easily verified that $h$ is an embedding.
\end{proof}

The embedding $\mathcal{L}(P) \to \mathcal{L}(Q)$ given by Corollary~\ref{cor:lattice-embedding} is not always meet-faithful. It is unclear whether requiring
that the embedding be meet-faithful would give a (much) better bound.

\begin{rem}
  Similarly to the extension complexity, the rectangle covering number is easily shown to be monotone on extensions and on faces. If the polytope $Q$
  is an extension of the polytope $P$, then we may infer, e.g., from Proposition~\ref{prop:embeddingFacLat} and Corollary~\ref{cor:lattice-embedding}
  that $\rc(Q) \geqslant \rc(P)$. Moreover, by reasoning on the slack matrices directly, we see that $\rc(P) \geqslant \rc(F)$ whenever polytope $P$ has $F$ as
  a face.
\end{rem}

\section{Upper Bounds on the Rectangle Covering Number}

In this section, we discuss some examples of polytopes for which small rectangle coverings can be found, as well as methods for constructing such rectangle coverings. These examples illustrate cases where the rectangle covering bound shows its limitations. Cases where the rectangle covering bound can be successfully applied to obtain interesting lower bounds on the \mnfx{} are discussed in the Section~\ref{sec:lb}.

\subsection{The Perfect Matching Polytope}

The \emph{perfect matching} polytope $\PMatch{n}$ is defined as the convex hull of characteristic vectors of perfect matchings in the complete graph $K^n = (V_n,E_n)$. A linear description of the perfect matching polytope $\PMatch{n}$ is as follows~\cite{Edmonds65}:
\begin{align*}\label{eq:PMPoly}
	\PMatch{n}=\{x\in\RR^{E_n}\,:\,
	\begin{array}[t]{rcl}
	x(\delta(v)) &= &1\text{ for all }v \in V_n,\\
	x(\delta(S)) &\geqslant &1\text{ for all } S \subseteq V_n, 3\leqslant |S| \leqslant n-3, |S| \text{ odd}\\
	x_e &\geqslant &0\text{ for all }e \in E_n\,\},
	\end{array}
\end{align*}
where $\delta(S)$ denotes all edges in the graph with exactly one endpoint in $S$, $\delta(v) := \delta(\{v\})$, and $x(F) := \sum_{e \in F} x_e$ for all edge sets $F \subseteq E_n$. Currently, no polynomial-size extension is known for the perfect matching polytope $\PMatch{n}$. Moreover, it was shown that under certain symmetry requirements no polynomial-size extension exists~\cite{Yannakakis91,KaibelPashkovichTheis10}. On the other side it is possible that a non-symmetric polynomial-size extension for the perfect matching polytope can be found~\cite{KaibelPashkovichTheis10}. In this context any non-trivial statement about the \mnfx{} of $\PMatch{n}$ is interesting.

The non-incidence matrix $M$ (w.r.t.\ the above system of constraints and the  characteristic vectors of perfect matchings) of the perfect matching polytope has $\Theta(n^2)$ rows corresponding to the non-negative constraints and $\Theta(2^n)$ rows corresponding to the inequalities indexed by odd subsets of vertices of $K^n$. A facet corresponding to some odd set $S$ and a vertex corresponding to some matching are non-incident if and only if the matching has more than one edge in $\delta(S)$.

\begin{lem}
The rectangle covering number  $\rc(\PMatch{n})$ is $O(n^4)$. 
\end{lem}

\begin{proof}
The non-zero entries in rows corresponding to the non-negative constraints can be trivially covered by $O(n^2)$ 1-rectangles of height one. Less obvious is the fact that the non-zero entries corresponding to the odd-subset inequalities can be covered by $O(n^4)$ 1-rectangles~\cite{Yannakakis91}. For this one considers the  rectangles $R_{e_1,e_2} = I_{e_1,e_2}\times J_{e_1,e_2}$ indexed by unordered pair of edges $e_1, e_2$, where the set $I_{e_1,e_2}$ consists of all odd sets $S$ such that $e_1$, $e_2\in\delta(S)$ and the set $J_{e_1,e_2}$ consists of all the matchings containing both edges $e_1, e_2$. 
\end{proof}

Thus in the case of $\PMatch{n}$ we cannot obtain any lower bound better than $O(n^4)$ solely by reasoning on rectangle coverings.

\subsection{Polytopes with Few Vertices on Every Facet}

We consider a polytope $P$ with $n$ vertices such that every facet of $P$ contains at most $k$ vertices, and its facet vs.\ vertex non-incidence matrix $M$. Each row of $M$ has at most $k$ zeros, and at least $n-k$ ones. We now prove that the rectangle covering number of such a polytope is necessarily small. 

\begin{prop}
\label{prop:UB_few_vtcs_per_facet}
{If $P$ is a polytope with $n$ vertices and at most $k$ vertices on each facet, then $\rc(P) = O(k^2 \log n)$.}
\end{prop}

{Before giving the proof, we state the following lemma that is the main ingredient of the proof.}

\begin{lem} \label{lem:cov_gen} Let $M$ be the 0/1 matrix with rows indexed by $\{S \subseteq \ints{n} : |S| \leqslant k\}$ and columns indexed by $\{T \subseteq \ints{n} : |T| \leqslant \ell\}$, where $\ell \leqslant k \leqslant n$, such that the entry $M_{S,T}$ is non-zero if and only if the sets $S$, $T$ are disjoint. Then $\rc(M) = O((k + \ell)e^{\ell}(1 + k/\ell)^{\ell} \log n)$.
\end{lem}
\begin{proof}
	Consider rectangles $I_V\times J_V$, $V\subseteq\ints{n}$, where $I_V :=\setDef{S\subseteq\ints{n}}{S\subseteq V, |S| \leqslant k}$ and $J_V :=\setDef{T\subseteq\ints{n}}{T\cap V=\varnothing, |T| \leqslant \ell}$.
	Obviously, every rectangle $I_V\times J_V$,  $V\subseteq\ints{n}$ is contained in $\supp(M)$.

	Pick a set $V\subseteq\ints{n}$ by selecting the points from $\ints{n}$ independently with probability $p$, for some $p \in [0,1]$. For a fixed pair $(S, T)$ of disjoint sets $S$, $T \subseteq\ints{n}$ with $|S| \leqslant k$ and $|T| \leqslant \ell$, the probability to be covered is at least 
	\begin{equation*}
		q:=p^k (1-p)^{\ell},
	\end{equation*}
	thus choosing the probability $p$ equal to $\frac{k}{k+\ell}$ to maximize the value $q$ we get
	\begin{equation*}
		q=\left(\frac{k}{k+\ell}\right)^{k}\left(1-\frac{k}{k+\ell}\right)^{\ell} \ge e^{-\ell} \left(\frac{\ell}{k+\ell}\right)^{\ell}\,.
	\end{equation*}

	Now, let us bound the natural logarithm of the expected number of entries from $\supp(M)$ which are not covered if we choose independently $r$ such rectangles. An upper bound on this quantity is:
	\begin{align*}
		\ln\left((n+1)^k (n+1)^\ell (1-q)^r\right) &\le (k+\ell) \ln(n+1) +r\ln(1-q)\\
		&\le (k+\ell) \ln(n+1) - r q\\
		&\le (k+\ell)\ln(n+1)-r e^{-\ell} \left(\frac{\ell}{k+\ell}\right)^{\ell}\,.	
	\end{align*}
	If the above upper bound for the logarithm of the expected number of not covered entries from $\supp(M)$ is negative, we can conclude that there exists a rectangle cover for the matrix $M$ of size $r$. Thus there exists a rectangle cover of size $ O((k + \ell)e^{\ell}(1 + k/\ell)^{\ell} \log n)$.
\end{proof}

\begin{proof}[Proof of Proposition \ref{prop:UB_few_vtcs_per_facet}]
{Let $M$ be the facet vs.\ vertex non-incidence matrix of $P$. It suffices to show that $\rc(M) = O(k^2 \log n)$.} We may extend $M$ by adding extra rows in order to obtain a matrix in which \emph{each} binary vector of size $n$ with at most $k$ zeros appears as a row. Obviously, this operation does not decrease $\rc(M)$. { The result then follows from Lemma~\ref{lem:cov_gen} by taking $\ell := 1$.}
\end{proof}
 
The most natural case where Proposition \ref{prop:UB_few_vtcs_per_facet} applies is perhaps when $P$ is a simplicial $d$-dimensional polytope with $n$ vertices. For such a polytope there is a rectangle covering of the non-incidence matrix with $O(d^2 \log n)$ rectangles.

\subsection{The Edge Polytope}

The \emph{edge polytope} $\PEdge{G}$ of a graph $G$ is defined as the convex hull of the incidence vectors in $\R^{V(G)}$ of all edges of $G$. Thus $\PEdge{G}$ is a 0/1-polytope with $|E(G)|$ vertices in $\R^{V(G)}$. Consider a stable set $S$ of $G$. Denoting by $N(S)$ the neighborhood of $S$ in $G$, we see that the inequality
\begin{equation}
\label{eq:P(G)_stable}
x(S) - x(N(S)) \leqslant 0
\end{equation}
is valid for $\PEdge{G}$. It can be shown \cite{KaibelLoos11} that these inequalities, together with { $x_v \geqslant 0$ for $v \in V(G)$} and $x(V(G))=2$, form a complete linear description of $\PEdge{G}$.

\subsubsection{A Sub-quadatric Size Extension for All Graphs}

The following lemma provides an upper bound not just on the rectangle covering number of $\PEdge{G}$, but also on the \mnfx{} of $\PEdge{G}$.

\begin{lem}
 For every graph $G$ with $n$ vertices there exists an extension of the edge polytope $\PEdge{G}$ of size $O(n^2/\log n)$.
\end{lem}
\begin{proof}
  Let $H$ denote a biclique (i.e., complete bipartite graph) with bipartition $W$, $U$. The edge polytope $P(H)$ of $H$ has a linear description of size
  $O(|V(H)|)$. Namely, in addition to the equalities $x(W)=1$ and $x(U)=1$, the linear description consists of the nonnegativity constraints $x_v \geqslant 0$
  for $v \in V(H)$. From any covering of the edges of $G$ by bicliques $H_1,\ldots, H_{t}$ (with $H_i \subseteq G$ for all $i$), we can construct an extension
  of $\PEdge{G}$ of the size $O(t+|V(H_1)|+\ldots+|V(H_t)|)$ using disjunctive programming \cite{Balas1985}. Since $t \leqslant |V(H_1)|+\ldots+|V(H_t)|$ and
  since there is such a covering with $|V(H_1)|+\ldots+|V(H_t)| = O(n^2/\log n)$ \cite{Tuza84} the lemma follows.
\end{proof}

\subsubsection{Dense Graphs}

For two  matrices $A$ and $B$ of the same size, denote by $A\centerdot B$ the entry-wise (or Hadamard) product: $(A\centerdot B)_{i,j} = A_{i,j}B_{i,j}$. The following statements are easy to verify.

\begin{lem} \label{lem:centerdot_nnegrank} \mbox{}
  \begin{enumerate}[(a)]
  \item For any nonnegative real matrices we have $\displaystyle \nnegrank(A\centerdot B) \leqslant \nnegrank(A) \, \nnegrank(B)$.
  \item For any real matrices we have  $\displaystyle \rc(A\centerdot B) \leqslant \rc (A) \,\rc (B)$.
  \end{enumerate}
\end{lem}

Let $G$ be a graph with $n$ vertices. We denote $\mathcal S(G)$ the set of all independent sets of $G$. The slack matrix of $\PEdge{G}$ has two type of rows: those which correspond to inequalities of the form \eqref{eq:P(G)_stable} for $S$ stable in $G$, and those which correspond to nonnegativity inequalities $x_v \geqslant 0$ for $v \in V(G)$. Let $M$ be the support of the submatrix of the slack matrix of $\PEdge{G}$ induced by the rows corresponding to stable sets $S$; let $M^{(1)}$ be the $\mathcal S(G) \times E(G)$-matrix with $M^{(1)}_{S,e} = 1$ if $e\cap S = \emptyset$ and $0$ otherwise; and let $M^{(2)}$ be the $\mathcal S(G) \times E(G)$-matrix with $M^{(2)}_{S,uv} = 1$ if $u$ or $v$ has a neighbor in $S$. Then we have
\begin{equation}\label{eq:factorize_graphslack}
  M = M^{(1)}\centerdot M^{(2)}.
\end{equation}

\begin{lem} \label{eq:cover_M_2}
 For the matrix $M^{(2)}$ defined above we have $\rc(M^{(2)}) \leqslant n$.
\end{lem}
\begin{proof}
  For each $v\in V(G)$, define a 1-rectangle $I_v\times J_v$. The set $I_v$ consists of all stable sets that contain a neighbor of $v$, and the set $J_v$ consists of all edges incident to $v$. These $n$ rectangles define a rectangle covering of $M^{(2)}$.
\end{proof}

Now, we can prove an upper bound on $\rc(P(G))$.

\begin{prop}\label{ub:edge-alpha}
  Denoting by $\alpha(G)$ the stability number of $G$, the support of the slack matrix of the edge polytope of $G$ admits a rectangle cover of size 
  \begin{equation*}
    O(\alpha(G)^3 n \log n).
  \end{equation*}
\end{prop}
{
\begin{proof}
Lemma \ref{lem:cov_gen} for $\ell := 2$ and $k := \alpha(G)$ implies $\rc(M^{(1)}) = O(\alpha(G)^3 \log n)$. The result follows by combining Lemmas \ref{lem:centerdot_nnegrank} and \ref{eq:cover_M_2}.
\end{proof}}

Combining Tur\'an's Theorem~\cite{Turan41}, which states that $|E(G)| \geqslant \frac{|V(G)|^2}{2 \alpha(G)}$, with Proposition \ref{ub:edge-alpha}, we obtain the final result of this section: there exists a nontrivial class of graphs for which the rectangle covering number cannot prove a strong $\Omega(|E(G)|)$ lower bound on the extension complexity.

\begin{cor}
  The support of the slack matrix of the edge polytope of $G$ admits a rectangle cover of size {$o(|E(G)|)$} if
  \begin{equation*}
    \alpha(G) = o\!\left( \bigl(\tfrac{n}{\log n}\bigr)^{1/4} \right).
    \vspace{-1em}
  \end{equation*}
  \qed
\end{cor}

\section{Rectangle Covering Number and Communication Complexity}\label{sec:CC}

The aim of \emph{communication complexity} is to quantify the amount of communication necessary to evaluate a function whose input is distributed among several players. Since its introduction by A.\ Yao in 1979~\cite{Yao1979}, it became a part of complexity theory. It was successfully applied, e.g., in the contexts of VLSI design, circuit lower bounds and lower bounds on data-structures~\cite{KushilevitzNisan97}. 

Two players Alice and Bob are asked to evaluate a Boolean function $f : A \times B \to \{0,1\}$ at a given input pair $(a,b) \in A \times B$, where $A$ and $B$ are finite sets. Alice receives the input $a \in A$ and Bob receives the input $b \in B$. Both players know the function $f$ to be evaluated, but none of the players initially has any information about the input of the other player. They have to cooperate in order to compute $f(a,b)$. Both players can perform any kind of computation. It is only the \emph{amount of communication} between them which is limited. In a deterministic protocol, Alice and Bob will exchange bits until one of them is able to correctly compute $f(a,b)$. The \emph{deterministic communication complexity} of the function $f$ is the minimum worst-case number of bits that Alice and Bob have to exchange in order to evaluate $f$ at any given input pair $(a,b)$. 

The function $f$ can be encoded via its \emph{communication matrix} $M = M(f)$ that has $M_{i,j} = f(a_i,b_j)$, where $a_i$ is the $i$-th element of $A$ and $b_j$ is the $j$-th element of $B$. It is known~\cite{Yannakakis91} that a deterministic protocol of complexity $k$ for $f$ yields a decomposition of $M$ as a sum of at most $2^k$ rank-one 0/1-matrices, implying $\nnegrank(M) \leqslant 2^k$. In order to apply this for constructing extensions, via Theorem~\ref{thm:extComplNonnegRk}, it is necessary for the slack matrix to be binary. 

Let $G$ be a graph. The \emph{stable set polytope} $\PStab{G}$ is the 0/1-polytope in $\R^{V(G)}$ whose vertices are the characteristic vectors of stable sets of $G$. Because the maximum stable set problem is NP-hard, it is unlikely that a complete description of $\PStab{G}$ will be found for all graphs $G$.  However, some interesting classes of valid inequalities for $\PStab{G}$ are known. For instance, for every clique $K$ of $G$, the \emph{clique inequality}
%
$\sum_{v \in K} x_v \leqslant 1$
%
is valid for $\PStab{G}$. By collecting the slack covectors of each of these inequalities w.r.t.\ the vertices of $\PStab{G}$, we obtain a 0/1-matrix $M = M(G)$. This matrix is the communication matrix of the so-called \emph{clique vs.\ stable set problem} (also known as the \emph{clique vs.\ independent set problem}). Yannakakis~\cite{Yannakakis91} found a $O(\log^2 n)$ complexity deterministic protocol for this communication problem. When the clique inequalities together with the non-negativity inequalities form a complete linear description of $\PStab{G}$, that is, when $G$ is perfect~\cite{Chvatal75}, one obtains an extension of size $2^{O(\log^2 n)} = n^{O(\log n)}$ for $\PStab{G}$, where $n$ denotes the number of vertices of $G$. For general graphs, $M$ is not a slack matrix of $\PStab{G}$ but a proper submatrix. Hence $\nnegrank(M)$ only gives a lower bound on the extension complexity of $\PStab{G}$.

Consider again a Boolean function $f : A \times B \to \{0,1\}$ and a deterministic protocol for computing $f$. Fix an input pair $(a,b)$. Given the transcript of the protocol on input pair $(a,b)$, each the players can verify independently of the other player if the part of the transcript that he/she is responsible for is correct. The transcript is globally correct if and only if it is correct for \emph{both} of the players. Thus one way to persuade each player that, say, $f(a,b) = 1$ is to give them the transcript for $(a,b)$. If the complexity of the given deterministic protocol is $k$, then the players can decide whether $f(a,b) = 1$ or not based on a binary vector of size $k$ (the transcript). This is an example of a nondeterministic protocol. We give a formal definition in the next paragraph.

In a \emph{nondeterministic protocol}, both players are given, in addition to their own input $a$ and $b$, a proof that is in the form of a binary vector $\pi$ of size $k$ (the same for both players). The players cannot communicate with each other. Each of them checks the proof $\pi$, independently of the other, and outputs one bit, denoted by $V_A(a,\pi)$ or $V_B(b,\pi)$. The result of the protocol is the \emph{AND} of these two bits, that is, $V(a,b,\pi) := V_A(a,\pi) \wedge V_B(b,\pi)$. The aim of the protocol is to prove that $f(a,b) = 1$. In order to be correct, the protocol should be \emph{sound} in the sense that if $f(a,b) = 0$ then $V(a,b,\pi) = 0$ for all proofs $\pi$, and \emph{complete} in the sense that if $f(a,b) = 1$ then $V(a,b,\pi) = 1$ for some proof $\pi$. The \emph{nondeterministic complexity of $f$} is the minimum proof-length $k$ in a nondeterministic protocol for $f$. Letting $M = M(f)$ denote the communication matrix of $f$, it is well-known that the nondeterministic complexity of $f$ is $\lceil \log(\rc(M)) \rceil$, see e.g.~\cite{KushilevitzNisan97}.

While Yannakakis proved that the deterministic communication complexity of the clique vs.\ stable set problem is $O(\log^2 n)$, the exact deterministic and nondeterministic communication complexities of the clique vs.\ stable set problem are unknown. The best result so far is a lower bound of $\frac{6}{5} \log n - O(1)$ on the nondeterministic complexity obtained by Huang and Sudakov~\cite{HuangSudakov2010}. They also made the following graph-theoretical conjecture that, if true, would improve this bound to $\Omega(\log^2 n)$. 

Recall that the \emph{biclique partition number} of a graph $G$ is the minimum number of bicliques (that is, complete bipartite graphs) needed to partition the
edge set of $G$. Recall also that the \emph{chromatic number} $\chi(G)$ of $G$ is the minimum number of parts in a partition of the vertex set of $G$ into
stable sets.
\begin{conj}[Huang and Sudakov~\cite{HuangSudakov2010}]
  For each integer $k > 0$, there exists a graph $G$ with biclique partition number $k$ and chromatic number at least $2^{c \log^2 k}$, for some constant $c >
  0$.
\end{conj}
This conjecture would settle the communication complexity of the clique vs.\ stable set problem. In terms of extensions, the result of Huang and Sudakov implies a $\Omega(n^{6/5})$ lower bound on the worst-case rectangle covering number of stable set polytopes of graphs with $n$ vertices. Moreover, if true, the Huang-Sudakov conjecture would imply a $n^{\Omega(\log n)}$ lower bound.

\section{Lower Bounds on the Rectangle Covering Number}\label{sec:lb}

In this section, we give lower bounds on the rectangle covering number, and apply them to prove results about the \mnfx{} of polytopes. Our general strategy is to focus on a specific submatrix of the slack matrix and then use simple structural properties of the support of the submatrix. Although much of the underlying geometry is lost in the process, we can still obtain interesting bounds, which are in some cases tight. Also, we compare the bounds to each other whenever possible and useful. We remark that most of the bounds discussed here are known bounds on the nondeterministic communication complexity of Boolean functions. Nevertheless, they were never studied in the context of polyhedral combinatorics.

\subsection{Dimension and Face Counting}\label{ssec:lb:dim-and-facenr}

\begin{lem}\label{lem:lb:log-distinct-rows}
  Suppose the 0/1-matrix $M$ has $h$ distinct rows. Then $\rc(M) \geqslant \log h$. 
\end{lem}
\begin{proof}
  Since deleting rows does not increase the rectangle covering number, we may assume that $M$ has exactly $h$ rows, which are all distinct.
    
  We proceed by contradiction, i.e., we assume that there is a collection $\mathcal{R}$ of rectangles that covers $M$ and contains less than $\log h$ rectangles.  Each such rectangle is of the form $I \times J$, where $I$ is a set of row indices and $J$ is a set of column indices such that $M_{i,j}=1$ whenever $i\in I$ and $j \in J$.  
  
  For every row-index $i\in\ints{h}$, let $\mathcal{R}_i$ denote the set of all rectangles $I\times J$ from $\mathcal{R}$ such that $i \in I$.  Since the number of rectangles in $\mathcal{R}$ is less than $\log h$, there are two distinct row-indices $i$ and $i'$ such that $\mathcal{R}_{i}$ and $\mathcal{R}_{i'}$ are equal. Because the rectangles in $\mathcal{R}$ form a covering of the matrix $M$, it follows that the $i$-th row and the $i'$-th row of $M$ have the same set of one-entries and are thus equal, a contradiction.
\end{proof}

Obviously, the same statement holds with ``rows'' replaced by ``columns''. It is well-known~\cite{Goemans09} that the binary logarithm of the number of faces of a polytope $P$ is a lower bound for the number of facets in an extension of $P$.  This number is in fact also a lower bound on the rectangle covering number of $P$. Furthermore, the following chain of inequalities holds.

\begin{prop}\label{prop:lb:facenr}
  For a full-dimensional polytope $P$ in $\mathbb{R}^d$ with $f$ faces, we have
  \begin{equation*}
    d + 1 \leqslant \log f \leqslant \rc(P).
  \end{equation*}
\end{prop}

\begin{proof} 
The first inequality follows by induction on $d$. For the second inequality, consider a slack matrix for $P$ with $f$ rows with the $i$-th row being the slack covector of any inequality defining the $i$-th face of $P$. The support matrix $M$ of this slack matrix has $f$ distinct rows. The conclusion now follows from
Lemma~\ref{lem:lb:log-distinct-rows}.
\end{proof}

\subsection{The Rectangle Graph}\label{ssec:lb:G(M)} Consider a real matrix $M$. The \emph{\thegraph} of $M$, denoted by $\theG{M}$, has the pairs $(i,j)$ such that $M_{i,j} \neq 0$ as vertices. Two vertices $(i,j)$ and $(k,\ell)$ of $\theG{M}$ are adjacent if no rectangle contained in $\supp(M)$ can be used to cover both $(i,j)$ and $(k,\ell)$. This last condition is equivalent to asking $M_{i,\ell} = 0$ or $M_{k,j} = 0$. The following result will allow us to interpret most lower bounds on $\rc(M)$ in terms of graphs. We omit the proof because it is straightforward. 

\begin{lem}
\label{lem:rc_chi}
For every real matrix $M$, we have $\rc(M) = \chi(\theG{M})$. \qed
\end{lem}

\subsection{Clique Number and Fooling Sets}\label{ssec:lb:clique}

The \emph{clique number} $\omega(\theG{M})$ of the \thegraph{} of a real matrix $M$ is a lower bound to its chromatic number and hence, by Lemma \ref{lem:rc_chi}, a lower bound on the rectangle covering number of $M$. A clique in $\theG{M}$ corresponds to what is known as a \emph{fooling set:} A selection of non-zeros in the matrix such that no two of them induce a rectangle contained in $\supp(M)$.  As is customary, we denote the clique number of a graph $G$ by $\omega(G)$. Analogously, we denote the clique number of $\theG{M}$ by $\omega(M)$. We can also define the clique number $\omega(P)$ of a polytope $P$ as $\omega(M)$, for any non-incidence matrix $M$ of $P$. The next proposition shows that $\omega(P)$ is well-defined.

\begin{prop}
The clique number $\omega(M)$ is independent of the non-incidence matrix $M$ chosen for the polytope $P$.
\end{prop}
\begin{proof}
Consider any non-incidence matrix $M$ for $P$ and a fooling set of size $q := \omega(M)$. It suffices to show that the facet-vertex non-incidence matrix of~$P$ has a fooling set of size~$q$ as well. By reordering rows and columns of $M$ if necessary, we may assume that the fooling set is $(F_1,G_1)$, \ldots, $(F_q,G_q)$. Thus, $F_1$, \ldots, $F_q$ and $G_1$, \ldots, $G_q$ are faces of $P$ satisfying: $F_i \not\supseteq G_i$ for all $i \in \ints{q}$, and $F_i \supseteq G_j$ or $F_j \supseteq G_i$ for all $i, j \in \ints{q}$ with $i \neq j$. We assume that our fooling set is chosen to maximize the number of $F_i$'s which are facets, plus the number of $G_i$'s which are vertices.

If every $F_i$ is a facet and every $G_i$ is a vertex, we are done because then our fooling set is contained in the facet vs.\ vertex submatrix of $M$. Otherwise, w.l.o.g., there exists an index $i$ such that $F_i$ is not a facet. Let $i'$ be such that $F_{i'}$ is a facet with $F_{i'} \supseteq F_i$ but at the same time $F_{i'} \not\supseteq G_i$. If $i' > q$ then we can replace $(F_i,G_i)$ by $(F_{i'},G_i)$ in the fooling set, contradicting the choice of the fooling set. Otherwise, $i' \leqslant q$ and the fooling set contains a pair of the form $(F_{i'},G_{i'})$. Since $F_{i'} \not\supseteq G_{i'}$, we have $F_{i} \not\supseteq G_{i'}$. Moreover, $F_{i'} \not\supseteq G_i$ by choice of $F_{i'}$. This contradicts the fact that $(F_i,G_i)$, \ldots, $(F_q,G_q)$ is a fooling set.
\end{proof}

In this section, by constructing large fooling sets, we can give lower bounds on the \mnfx{} for a number of examples, including cubes (in Section
\ref{ssec:lb:cube}) and the Birkhoff polytope (in Section~\ref{ssec:lb:birkhoff}).  But before we do that, we will take a look at the limitations of the clique
number as a lower bound for the rectangle covering number.  We show that the clique number of the \thegraph{} is always $O(d^2)$ for $d$-dimensional polytopes,
and $O(d)$ for simple polytopes.

\subsubsection{Limitations of the Clique Number as a Lower Bound on the Rectangle Covering Number}\label{sssec:lb:clique:limit}

Here we give some upper bounds on the sizes of fooling sets. In some cases, these bounds immediately render useless the fooling set approach to obtain a
desired lower bound for the rectangle covering number or the \mnfx{} of a polytope.
We start with an easy upper bound based on the number of zeros per row.

\begin{lem}\label{lem:lb:clique:zeros}
  If every row of $M$ contains at most $s$ zeros, then $\omega(M) \leqslant 2s+1$.
\end{lem}
\begin{proof}
If $M$ has a fooling set of size $q$, then the submatrix of $M$ induced by the rows and columns of the entries belonging to the fooling set contains at least $q(q-1)/2$ zeros. In particular, one row of the submatrix has at least $(q-1)/2$ zeros. By hypothesis, $(q-1)/2 \leqslant s$, that is, $q \leqslant 2s + 1$.
\end{proof}

If every vertex of a polytope is contained in at most $s$ facets, or if every facet of a polytope contains at most $s$ vertices, then $\omega(P) = O(s)$.  In particular, for simple or simplicial polytopes, the fooling set lower bound is within a constant factor of the dimension.

We now give dimensional upper bounds on $\omega(M)$.  Dietzfelbinger et al.~\cite{DietzfelbingerHromkovicSchnitger96} show the following.

\begin{lem}[Dietzfelbinger et al.~\cite{DietzfelbingerHromkovicSchnitger96}]
  For every field $\mathbb{K}$ and for every 0/1-matrix $M$, there is no fooling set larger than $\rk_{\mathbb{K}}(M)^2$.
\end{lem}

{Although the rank of the non-incidence matrix $M$ of a polytope can be substantially larger than its dimension (this is the case already for polygons), we can nevertheless prove the following by techniques similar to those of \cite{DietzfelbingerHromkovicSchnitger96}.}

\begin{lem}\label{lem:lb:clique:tensor}
  For every polytope $P$ of dimension $d$, we have $\omega(P) \leqslant (d+1)^2$.
\end{lem}
\begin{proof}
  { Without loss of generality, assume that $P$ is full-dimensional. Let $M$ be a non-incidence matrix for $P = \{x \in \R^d : Ax \leqslant b\} = \conv(\{v_1,\ldots,v_n\})$, and $q := \omega(P)$. By reordering if necessary, we may assume that $(1,1),\dots, (q,q)$ are the vertices of a maximum clique of the \thegraph{} of $M$.  Let $z_i := (b_i, -A_{i,\star}) \in \RR^{d+1}$, $i=1,\dots,q$, and $t_j := (1,v_j) \in \RR^{d+1}$, $j=1,\dots,q$.}  These vectors have the following properties:
\begin{equation}\label{eq:lb:geom-fool}
    \begin{aligned}
      \scalProd{z_i}{t_j} &\geqslant 0 &&\text{ for all } i,j = 1,\dots,q\,; \\
      \scalProd{z_i}{t_i} &> 0 &&\text{ for all } i = 1,\dots,q\,;\\
      \scalProd{z_i}{t_j} &= 0 &&\text{ or } \quad\scalProd{z_j}{t_i} = 0 \text{ for all } i,j = 1,\dots,q \text { with } i \neq j\,.
    \end{aligned}
  \end{equation}
  
  Now consider the following $2q$ rank-one matrices: $z_i t_i^T$ for $i=1,\dots,q$ and $t_j z_j^T$ for $j=1,\dots,q$.  Taking the usual inner product for matrices $\scalProd{A}{B} = \sum_{i,j} A_{ij}B_{ij} = \mathrm{Tr}(A^T B)$, if $i \neq j$, we have 
  \begin{equation*}
    \scalProd{z_i t_i^T}{t_j z_j^T} = \mathrm{Tr}(t_i z_i^T t_j z_j^T) = \mathrm{Tr} (z_j^T t_i z_i^T t_j) = (z_j^T t_i) (z_i^T t_j)  =  
    \scalProd{z_j}{t_i} \scalProd{z_i}{t_j} = 0\,.
  \end{equation*}
  but
  \begin{equation*}
    \scalProd{z_i t_i^T}{t_i z_i^T} = \scalProd{z_i}{t_i} \scalProd{z_i}{t_i} > 0\,.
  \end{equation*}
  This implies that the matrices $z_1 t_1^T, \ldots, z_q t_q^T$ are  linearly independent.  Since we
  have $q$ linearly independent $(d+1)\times(d+1)$-matrices, we conclude that
  $q \leqslant (d+1)^2$.
\end{proof}

\begin{rem}
  From a construction due to Dietzfelbinger et al.~\cite{DietzfelbingerHromkovicSchnitger96}, we can infer the existence of a $d$-dimensional polytope $P$ with
  $\omega(P) = \Omega(d^{\frac{\log 4}{\log 3}})$.  It is an open question which of these bounds can be improved. \end{rem}

We conclude this subsection with an example where the fooling set or clique number bound is particularly bad.  Consider the vertex-facet non-incidence-matrix of a
convex polygon with {$n$} vertices.  This is a $n \times n$ 0/1-matrix $M$ with $M_{i,j} = 0$ if and only if $i$ is equal to $j$ or $j+1$ modulo $n$.  We have $\omega(M) \leqslant 5$, by Lemma~\ref{lem:lb:clique:zeros}, whereas $\rc(M) \geqslant \log n$ by Proposition~\ref{prop:lb:facenr}.

\subsubsection{The Cube}\label{ssec:lb:cube}

We now apply the fooling set technique to show that $d$-cubes are ``minimal extensions'' of themselves.  A \emph{combinatorial $d$-cube} is a polytope whose face lattice is isomorphic to that of {the unit cube} $[0,1]^d$.  Compare the result of this proposition to the lower bound one can obtain as the binary logarithm of the number of faces, Proposition~\ref{prop:lb:facenr}, which is $d \log(3) \approx \mbox{1.585}\,d$.

\begin{prop}\label{prop:lb:clique:cube}
  If $P$ is a combinatorial $d$-cube then $\xc{P} = 2d$.
\end{prop}
\begin{proof} 
{ We obviously have $\omega(P) \leqslant \xc{P} \leqslant 2d$. Below, we prove that equality holds throughout by constructing a fooling set of size $2d$, implying $\omega(P) \geqslant 2d$. Since $\omega(P)$ only depends on the combinatorial type of $P$, we may assume that $P = [0,1]^d$.}  To define { our} fooling set, we carefully select for each facet-defining inequality of $P$ a corresponding vertex. The vertex corresponding to the inequality $x_i\geqslant 0$ is the vertex $v_i$ defined as follows:
  \begin{equation*}
    (v_{i})_{j} :=
    \left\{\begin{array}{ll}
        1& \textrm{if}\quad  1\leqslant j\leqslant i\\
        0& \textrm{if}\quad i<j\leqslant d
      \end{array}\right.
  \end{equation*}
  For the inequality $x_i\leqslant 1$, take the vertex $w_i$ defined as follows:
  \begin{equation*}
    (w_{i})_{j} :=
    \left\{\begin{array}{ll}
        0& \textrm{if}\quad  1\leqslant j\leqslant i\\
        1& \textrm{if}\quad i<j\leqslant d
      \end{array}\right.
  \end{equation*}
  One can check that none of the facets of the cube is incident with its corresponding vertex, but that for any two facets at least one of the corresponding vertices is incident to one of them.
\end{proof}

\subsubsection{The Birkhoff Polytope}\label{ssec:lb:birkhoff}

The $n$th \emph{Birkhoff polytope} is the set of doubly stochastic $n\times n$ matrices or, equivalently, the convex hull of all $n \times n$ permutation matrices. Let $P$ denote the $n$th Birkhoff polytope. For $n=1$, $P$ is a point (and $\xc{P}$ is not defined). For $n = 2$, $P$ is a segment and therefore $\xc{P} = \omega(P) = 2$. For $n = 3$,  $P$ has $9$ facets and $6$ vertices, thus $\xc{P} \leqslant 6$. As is easily verified, the facet vs.\ vertex non-incidence matrix for $P$ has a fooling set of size $6$. Hence, $\xc{P} = \omega(P) = 6$.

\begin{prop}\label{prop:lb:birhoff}
  For $n\geqslant 4$, the \mnfx{} of the Birkhoff polytope is $n^2$.
\end{prop}
\begin{proof}
The Birkhoff polytope $P$ is defined by the nonnegativity inequalities $x_{i,j} \geqslant 0$, $i,j\in[n]_0=\{0,\dots,n-1\}$ (starting indices from zero will be a bit more convenient for our treatment), and the equations $\sum_{i} x_{i, j} = 1$, $j\in[n_0]$, $\sum_j x_{i,j} =1$, $i\in[n_0]$. This trivially implies that $\xc{P} \leqslant n^2$.
  To give a lower bound on the \mnfx{}, we construct a fooling set of size $n^2$ in the facet vs.\ vertex non-incidence matrix of~$P$.

  For this, for every inequality $x_{i,j} \geqslant 0$, we define one vertex $v_{i,j}$ such that $(v_{i,j})_{i,j} > 0$ by giving a permutation $\pi$ of $[n]_0$.
  Namely, let $\pi(i):=j$ and $\pi(i+1):=j+1$ (all indices are to be understood modulo $n$).  Moreover, we take $\pi(k) := i+j+1-k$ whenever $k \not\in \{i,i+1\}$. This defines a permutation, and thus a vertex $v_{i,j}$ of $P$.

  Now, we show that this family of inequality-vertex pairs is a fooling set.  By contradiction, let us assume that for two different inequalities $x_{i,j} \geqslant 0$ and $x_{i',j'} \geqslant 0$ we have both $(v_{i,j})_{i',j'} = 1$ and $(v_{i',j'})_{i,j} = 1$. Letting $\pi$ and $\pi'$ denote the permutations of $[n]_0$ associated to $v_{i,j}$ and $v_{i',j'}$ respectively, this means that $\pi(i') = j'$ and $\pi'(i) = j$. Because $(i,j) \neq (i',j')$, we conclude that $i+j+1-i' = j'$ or $(i',j') = (i+1,j+1)$, and at the same time $i'+j'+1-i = j$ or $(i,j) = (i'+1,j'+1)$. Because $2 \neq 0$ and $3 \neq 0$ modulo $n$, these conditions lead to a contradiction.
\end{proof}

The Birkhoff polytope can also be described as the convex hull of the characteristic vectors of all perfect matchings of a complete bipartite graph with $n$ vertices on each side of the bipartition. The $n$th \emph{bipartite matching polytope} is the convex hull of all (not necessarily perfect) matchings of a complete bipartite graph with $n$ vertices on each side. Let $P$ denote the $n$th bipartite matching polytope. For $n = 1$, $P$ is a segment and $\xc{P} = \omega(P) = 2$. For $n = 2$, $P$ has $8$ facets and $7$ vertices, thus $\xc{P} \leqslant 7$. We leave to the reader to check that $\omega(P) \geqslant 7$, thus $\xc{P} = \omega(P) = 7$. For $n = 3$, $P$ has $15$ facets and $34$ vertices, thus $\xc{P} \leqslant 15$. It can be checked that $P$ has a fooling set of size $15$, hence $\xc{P} = \omega(P) = 15$. (To the inequality $x_{i,j} \geqslant 0$ we associate the matching $\{(i,j),(i+1,j+1)\}$ if $i \neq j$ and $\{(i,j),(k,\ell),(\ell,k)\}$ if $i = j$ and $k, \ell$ denote the two other elements of $\{0,1,2\}$. To the inequality $\sum_j x_{i,j} \leqslant 1$ we associate the matching $\{(i+1,i+1),(i+2,i)\}$. To the inequality $\sum_i x_{i,j} \leqslant 1$ we associate the matching $\{(j+1,j+1),(j,j+2)\}$.) As before, all computations are done modulo $n=3$.) For $n \geqslant 4$, an argument similar to the one in the proof of Proposition~\ref{prop:lb:birhoff} shows the following.

\begin{prop}
  For $n \geqslant 4$, the bipartite matching polytope has \mnfx{} $n^2+2n$.
\end{prop}
\begin{proof}[Proof (sketch).]
  Again, the given number is a trivial upper bound.  To construct a fooling set of this size, for the nonnegativity inequalities, we take the perfect matchings
  constructed in the proof of the previous proposition.  In addition, for the inequalities $\sum_{i} x_{i,j} \leqslant 1$, we take the vertex $w_j$ with $(w_j)_{k+1,k} = 1$ for all $k \ne j$; for the inequalities $\sum_{j} x_{ i,j} \leqslant 1$, we take the vertex $u_i$ with $(u_i)_{k,k+1} = 1$ for all $k \ne i$.
\end{proof}

\subsection{Independence Ratio, Rectangle Sizes, and Generalized Fooling Sets}\label{ssec:lb:alpha-goof}

Denote by $\alpha(G)$ the maximum cardinality of an independent vertex set in $G$.  The number $\frac{\nvtx{G}}{\alpha(G)}$ is a lower bound on the chromatic
number ($\nvtx{G}$ stands for the number of vertices in~$G$.)   Moreover, taking induced subgraphs may improve the bound: The number
\begin{equation*}
  \iota(G) := \max_{U \subseteq V(G) \atop U \neq \varnothing}\frac{\nvtx{U}}{\alpha(G[U])}\,,
\end{equation*}
where $G[U]$ denotes the subgraph of $G$ induced by the vertices in $U$ is sometimes called the \emph{independence ratio} of $G$, and is also a lower bound on the chromatic number of $G$.

In the context of the rectangle covering number, these bounds are known under different names. The following lemmas will make that clear. We leave the easy proofs to the reader.

\begin{lem}
  For a real matrix $M$, the maximum number of entries in a rectangle contained in $\supp(M)$ equals $\alpha(\theG{M})$. \qed
\end{lem}
%

The concept of \emph{generalized fooling sets} has been proposed by Dietzfelbinger et al.~\cite{DietzfelbingerHromkovicSchnitger96} as a lower bound on the nondeterministic communication complexity of a Boolean function.  In the graph-coloring terminology, a \emph{$k$-fooling set} is an induced subgraph~$H$ of $\theG{M}$ for which $\alpha(H) \leqslant k$ holds.  If a $k$-fooling set on $s$ entries of $\supp(M)$ can be found, then, clearly, $s/k$ is a lower bound on the rectangle covering number.  Their \emph{generalized fooling set lower bound} on the rectangle covering number is then just the supremum of all these fractions $s/k$, and it coincides with the independence ratio.

Whereas the fooling set / clique lower bound can be arbitrarily bad (see the example at the end of Section~\ref{sssec:lb:clique:limit}), this is not the case for the independence ratio.  The following is true for general graphs (Lov\'{a}sz~\cite[Thm.~7]{Lovasz_greedy}). It has been rediscovered  by Dietzfelbinger et al.~\cite{DietzfelbingerHromkovicSchnitger96} for the rectangle graphs of 0/1-matrices. Lov\' {a}sz' argument (specialized from the more general setting of hypergraph coverings he in fact considers) proceeds by analyzing the following greedy heuristic for coloring $G$: pick a maximum stable set $S_1$ in $G$, then a maximum stable set $S_2$ in $G-S_1$, and so on. Denoting by $w_k$ the maximum number of nodes in~$G$ containing no stable set of size larger than~$k$ and by $t_i$ the number of stable sets of exactly size~$i$ produced by the greedy heuristic, one finds that $\sum_{i=1}^kit_i\le w_k$ holds for all~$k=1,\dots,\alpha=\alpha(G)$. Adding up these inequalities scaled by $1/k(k+1)$ for $k<\alpha$ and by $1/\alpha$ for $k=\alpha$ one obtains the upper bound $\sum_{i=1}^{\alpha-1}\tfrac{1}{i(i+1)}w_i+\tfrac{1}{\alpha}w_{\alpha}=O\bigl( \iota(G)\log(\nvtx{G}) \bigr)$ on the number $\sum_{i=1}^{\alpha}t_i$ of stable sets generated by the greedy procedure. 

\begin{lem}
  For all graphs $G$, we have
  $\displaystyle \chi(G) = O\bigl( \iota(G)\log(\nvtx{G}) \bigr)$.
  \qed
\end{lem}

\subsubsection{The Cube Revisited}

We now give an alternative proof of Proposition~\ref{prop:lb:clique:cube}, based on the independence number of $\theG{P}$ instead of the clique number.

\begin{proof}[2nd Proof of Proposition~\ref{prop:lb:clique:cube}]
  The maximal {rectangles} contained in $\supp(M)$ of the facet vs.\ vertex non-incidence matrix $M$ are among those of the form $I \times J$, where $I$ is a collection of facets of $P$ that does not contain a pair of opposite facets and $J = J(I)$ is the set of those vertices of $P$ that belong to none of the facets in $I$. If $q$ denotes the
  cardinality of $I$, the size of such a rectangle is precisely $q \cdot 2^{d-q}$. Indeed, $q$ of the coordinates of the vertices of $J$ are determined
  by $J$, while the other $d-q$ coordinates are free, which implies $|J| = 2^{d-q}$.
  
  As is easily verified, $q 2^{d-q}$ is maximum for $q \in \{1,2\}$. Hence, the maximum size of a {rectangle} in $\supp(M)$ is thus $2^{d-1}$. It follows that $2d$ rectangles are necessary to cover all the ones in $M$, because there are $2d \cdot 2^{d-1}$ ones in $M$.
\end{proof}

\subsubsection{Polytopes with Few Vertices on Every Facet}

As we have seen in Lemma~\ref{lem:lb:clique:zeros} above, the chances of obtaining a good lower bound based on fooling sets are poor if the matrix has few zeros in each row. A similar result is true for generalized fooling sets.

\begin{lem}
  Let $M$ be a real matrix with at most $s$ zeros per row and let $U$ be a $k$-fooling set in $M$.  Then $|U|/k$, namely, the lower bound on $\rc(M)$ given by $U$, is at most $2k^2s$.
\end{lem}
\begin{proof}
  Let $U$ be a set of vertices of $\theG{M}$ which is a $k$-fooling set, i.e. the subgraph $H$ of $\theG{M}$ induced by $U$ has $\alpha(H) \leqslant k$.  By Tur\'an's theorem, we have
  \begin{equation*}  
    |E(H)| \geqslant \frac{|U|^2}{2k}\,.
  \end{equation*}

  Each edge of $H$ links two entries of $M$ in distinct rows and columns such that at least one of the two other entries in the $2 \times 2$ rectangle spanned by these entries is zero. We say that this (these) zero(s) are \emph{responsible} for the edge. Thus for each edge of $H$ there is at least one responsible zero (and at most two responsible zeros), located in the rectangle spanned by $U$. Since every row or column contains at most $k$ elements of $U$, each zero in the rectangle spanned by $U$ is responsible for at most $k^2$ edges of~$H$. Hence, the number of edges in $H$ is at most $k^2$ times the number of zeros in the rectangle spanned by $U$. Thus, since $U$ covers at most $\lvert U\rvert$ rows, which together contain at most $\lvert U\rvert s$ zeros, we have
  \begin{equation*}
    |E(H)| \leqslant k^2 |U| s\,,
  \end{equation*}
  and we conclude that 
  \begin{equation*}
    \frac{\lvert U \rvert}{k} \leqslant 2k^2 s\,.
  \end{equation*}
\end{proof}

\subsection{Fractional Chromatic Number}\label{sec:fracChrom}

Another well-studied lower bound for the chromatic number of a graph is the \emph{fractional} chromatic number. This number can be defined in several equivalent ways. The most convenient is probably the following:  $\chi^*(G)$ is the solution of the following linear program:
\begin{subequations}\label{eq:lb:chi-star}
  \begin{equation}
    \min \sum_S x_S \\
  \end{equation}
  \begin{align}
      \sum_{S\colon v\in S} x_S &\geqslant 1 \quad\text{for every vertex } v \text{ of } G\\
      x_S &\geqslant 0 \quad\text{for every stable set } S \text{ of } G,
    \end{align}
\end{subequations}
where the sums extend over all stable sets $S$ of the graph $G$.  In the context of nondeterministic communication complexity, this bound is
well-known, see for example~\cite{KarchmerKushilevitzNisan95} or \cite{KushilevitzNisan97} and the references therein.  In particular, it is known that $\chi(M) = O(\chi^*(\theG{M} \log \nvtx{\theG{M}}))$.  The following known fact from graph coloring (see e.g.~\cite[p.\ 1096]{SchrijverBookB03}) improves this slightly.

\begin{lem}
  For all graphs $G$, it is true that $\displaystyle \chi(G) \leqslant (1+\ln \alpha(G))\chi^*(G)$.
  \qed
\end{lem}

Clearly, if integrality conditions are imposed on the variables in the linear program~\eqref{eq:lb:chi-star} defining $\chi^*(G)$, then the optimal value is the
chromatic number $\chi(G)$.  

\subsection{Neighborly $d$-Polytopes with $\Theta(d^2)$ Vertices}\label{ssec:lb:hypergraph-vertex-cover}

Suppose $d$ is even and let $k := d/2$. Consider a neighborly polytope $P$ in $\mathbb{R}^d$ with vertex set $V := \vertexset(P)$. Let $n := |V|$ denote the number of vertices of $P$. We assume that $n = \Theta(d^2) = \Theta(k^2)$, that is, $k$ and $d$ are both $\Theta(\sqrt{n})$.

Consider a non-incidence matrix that has one row for each facet, as well as one row for each $(k-1)$-face, and one row per vertex. Let $M$ denote the submatrix induced on the rows that correspond to the $(k-1)$-faces. Thus the rows of $M$ are indexed by the $k$-subsets of $V$ (that is, the vertex sets of $(k-1)$-faces of $P$), and the columns of $M$ are indexed by the elements of $V$ (that is, the vertices of $P$). The entry of $M$ for $k$-set $F \subseteq V$ and vertex $v \in V$ is $1$ if $v \notin F$ and $0$ otherwise. It follows from Lemma~\ref{lem:rc_independent_of_S} and the monotonicity of the rectangle covering number on submatrices that $\xc{P} \geqslant \rc(M)$.

\begin{prop}\label{prop:lb:neigh}
  Let $n, k$ be positive integers and $M$ denote the ${n \choose k} \times n$ matrix defined above. Then $\rc(M) = \Omega\lt(\min\left\{n,\frac{(k+1)(k+2)}{2}\right\}\rt)$.
\end{prop}

\begin{proof}
  The maximal rectangles of $M$ are of the form $I \times J$, where $J$ is a subset of $V$ of size at most $n-k$ and $I = I(J)$ is
  the collection of all $k$-sets $F$ such that $F \subseteq V - J$. We define a \emph{cover} as a collection $\mathcal{J} = \{J_1, \ldots, J_t\}$ of subsets of
  $V$, each of size at most $n-k$, such that the corresponding rectangles cover $M$. This last condition can be restated as follows: for each pair $(F,v)$ with
  $v \notin F$ there exists an index $\ell \in \ints{t}$ with $v \in J_\ell$ and $F \subseteq V - J_\ell$.

  Consider a cover $\mathcal{J} = \{J_1, \ldots, J_t\}$. For a vertex $v \in V$, consider the collection $\mathcal{J}(v)$ of sets in $\mathcal{J}$ that contain $v$. There are two cases. First, it could be that $\{v\} \in \mathcal{J}(v)$. In this case, by removing $v$ from the other sets in $\mathcal{J}(v)$, we may
  assume that $\mathcal{J}(v) = \{\{v\}\}$. That is, the other sets of $\mathcal{J}$ do not contain $v$. Otherwise, all sets in $\mathcal{J}(v)$ contain at least
  one element distinct from $v$. In this case, it should not be possible to find a $k$-set $F$ contained in $V - \{v\}$ that meets all sets in $\mathcal{J}(v)$,
  because otherwise $\mathcal{J}$ would not cover the pair $(F,v)$. In particular, there are at least $k+1$ sets in $\mathcal{J}(v)$.

  Let $X$ denote the set of the vertices $v$ such that $\{v\} \in \mathcal{J}$, and let $s := |X|$. Pick distinct vertices $v_1$, $v_2$, \ldots, $v_{k+1}$ in $V - X$. This can be assumed to be possible because otherwise $V - X$ has at most $k$ vertices. If $V - X$ is empty, then clearly the cover is trivial and $t \geqslant n$. Else $V - X$ contains an element $v$ and, because all sets in $\mathcal{J}(v)$ are contained in $V - X$, one can easily find a $k$-set $F \subseteq V - \{v\}$ such that $(F,v)$ is not covered by $\mathcal{J}$. 
  
  For all $\ell \in \ints{k+1}$, it holds that the number of sets in $\mathcal{J}$ that contain $v_\ell$ and none of the vertices $v_1$, \ldots, $v_{\ell-1}$ is at least $k+2-\ell$. Otherwise, we could find $k+1-\ell$ elements in $V - \{v_1,\ldots,v_{\ell}\}$ that, added to $\{v_1,\ldots,v_{\ell-1}\}$, would define a $k$-set $F \subseteq V - \{v_\ell\}$ such that $(F,v_\ell)$ is not covered by $\mathcal{J}$.

  Now, the number of sets in $\mathcal{J}$ is at least
  \begin{equation*}
    s + (k+1) + k + (k-1) + \ldots + 1 = s + \frac{(k+1)(k+2)}{2} \geqslant \frac{(k+1)(k+2)}{2}\,.
  \end{equation*}
  The first term counts the number of singletons in $\mathcal{J}$, the second the number of sets containing $v_1$, the third the number of sets containing $v_2$
  but not $v_1$, and so on.
\end{proof}

When $n = \Theta(k^2)$, it follows from Proposition \ref{prop:lb:neigh} that the minimum number of rectangles needed to cover $M$ is $\Theta(n)$. Therefore, in this case the minimum number of facets in an extension of $P$ is $\Theta(n)$. Note that the binary logarithm of the total number of faces is in this case $\Theta(d \log n) = \Theta(d \log d)$. In conclusion, neighborly $d$-polytopes with $\Theta(d^2)$ vertices give a family of polytopes such that the rectangle covering number is super-linear in both the ``dimension'' and the ``binary logarithm of the number of faces'' bounds.

\section{Concluding Remarks}

We conclude this paper with some open problems. We begin by reiterating the first open problem in Yannakakis's paper~\cite{Yannakakis91}. Bounding the rectangle covering number seems to be the  currently best available approach to bound the extension complexity of specific polytopes. (The new bound on the nonnegative rank proposed by Gillis and Glineur~\cite{GillisGlineur10} does not lead to improvements for slack matrices.) Can one find other bounds? In particular, for a given polytope~$P$, consider the smallest number of facets of a polytope~$Q$ into whose face lattice the face-lattice of~$P$ can be embedded meet-faithfully.  Does this number improve substantially on the rectangle covering bound?  Or is it always bounded by a polynomial in the rectangle covering bound for~$P$?

For some polytopes the rectangle covering number even yields optimal or near-optimal lower bounds on the extension complexity as we have demonstrated in this paper for cubes, Birkhoff polytopes, and neighborly $d$-polytopes with $\Theta(d^2)$ vertices. However, for other polytopes such as the perfect matching polytope, the rectangle covering bound seems rather useless. For instance, for sufficiently irregular $n$-gons in~$\R^2$, the extension complexity is bounded from below by $\Omega(\sqrt{n})$, while the rectangle covering number is $O(\log n)$~\cite{FioriniRothvossTiwary11}.

In all the examples we have found so far, the rectangle covering number is always polynomial in $d$ and $\log n$, where $d$ denotes the dimension and $n$ the number of vertices. Can one find polytopes for which the rectangle covering number is super-polynomial in $d$ and $\log n$?

As mentioned above, it is true that $\omega(M) \leqslant(\rk A+1)^2$ for every matrix~$A$ whose support is~$M$, and Dietzfelbinger et
al.~\cite{DietzfelbingerHromkovicSchnitger96} constructed a family of matrices~$M$ with $\rk(M)^{\log 4/\log 3} \leqslant\omega(M)$.  Which of the two bounds can be
improved?

In view of our result on the extension complexity of neighborly $d$-polytopes with~$n$ vertices, it is natural to ask whether the bound $\Omega(d^2)$ can be
improved for $d = o(\sqrt n)$.

Finally, in the absence of better lower bounds on the extension-complexity, it would be interesting to know the rectangle covering bound of the perfect matching polytope.  As mentioned above, it is $O(n^4)$.  Even a small improvement on the trivial lower bound  $\Omega(n^2)$ would be interesting.

\section{Final Note}

After submitting this paper, our third open problem was solved by Fiorini, Massar, Pokutta, Tiwary and de~Wolf \cite{FMPTW12}. They prove among other things that the rectangle covering number of the cut polytope is super-polynomial in the dimension and logarithm of the number of vertices of the polytope.

\section*{Acknowledgments}

We are greatful to the referees whose comments lead to significant improvements in the presentation of the material.

\input{bibliography}
\end{document}

%% file: ext_ex.eps_t
\begin{picture}(0,0)%
\epsfig{file=ext_ex.eps}%
\end{picture}%
\setlength{\unitlength}{3947sp}%
\begingroup\makeatletter\ifx\SetFigFont\undefined%
\gdef\SetFigFont#1#2#3#4#5{%
  \reset@font\fontsize{#1}{#2pt}%
  \fontfamily{#3}\fontseries{#4}\fontshape{#5}%
  \selectfont}%
\fi\endgroup%
\begin{picture}(1958,1894)(3438,-2321)
\put(4944,-2276){\makebox(0,0)[lb]{\smash{{\SetFigFont{12}{14.4}{\sfdefault}{\mddefault}{\updefault}{\color[rgb]{0,0,0}$P$}%
}}}}
\put(3438,-1687){\makebox(0,0)[lb]{\smash{{\SetFigFont{12}{14.4}{\sfdefault}{\mddefault}{\updefault}{\color[rgb]{0,0,0}$\pi$}%
}}}}
\put(4827,-1521){\makebox(0,0)[lb]{\smash{{\SetFigFont{12}{14.4}{\sfdefault}{\mddefault}{\updefault}{\color[rgb]{0,0,0}$Q$}%
}}}}
\end{picture}%